\documentclass[12pt]{amsart} 

\usepackage{amssymb}
\usepackage{amsmath}
\usepackage{amsthm}
\usepackage{verbatim}
\usepackage{url}
\usepackage{color}
\usepackage{graphicx}
\usepackage{todonotes}
\usepackage{tikz}
\usepackage{tikz-cd}
\usepackage{pgfplots}
\pgfplotsset{compat=1.16}
\usepackage{float}

\usepackage{array}
\usepackage{booktabs}

\renewcommand{\r}{\rho}

\newcommand{\FF}{\mathbb{F}}

\newcommand{\NN}{\mathbb{N}}
\newcommand{\QQ}{\mathbb{Q}}
\newcommand{\RR}{\mathbb{R}}
\newcommand{\ZZ}{\mathbb{Z}}

\newcommand{\mcm}[3]{\newcommand{#1}[#2]{{\ensuremath{#3}}}} 
\mcm{\restric}{0}{\upharpoonright}

\numberwithin{equation}{section}

\newtheorem{theorem}[equation]{Theorem}
\newtheorem{lemma}[equation]{Lemma}

\newtheorem{proposition}[equation]{Proposition}

\newtheorem{corollary}[equation]{Corollary}

\newtheorem*{theorem*}{Theorem}

\newtheorem{thmA}{Theorem} 

\theoremstyle{definition}
\newtheorem{example}[equation]{Example}

\newtheorem{definition}[equation]{Definition}

\theoremstyle{remark}
\newtheorem{remark}[equation]{Remark}

\begin{document}
\title[Necklaces, permutations, and periodic critical orbits]{Necklaces, permutations, and periodic critical orbits for quadratic polynomials}
\author{Matthew Baker, Andrea Chen, Sophie Li, Matthew Qian}
\address{School of Mathematics, Georgia Institute of Technology, USA}
\email{mbaker@math.gatech.edu}
\address{Department of Mathematics, Massachusetts Institute of Technology, USA}
\email{andreayc@mit.edu}
\address{Department of Mathematics, Columbia University, USA}
\email{sql2002@columbia.edu}
\address{Department of Mathematics, Harvard University, USA}
\email{mqian@college.harvard.edu}

\date{\today}


\begin{abstract}
Let $G_n$ denote the $n^{\rm th}$ {\em Gleason polynomial}, whose roots correspond to parameters $c$ such that the critical point $0$ is periodic of exact period $n$ under iteration of $z^2 + c$, and let $\bar{G}_n$ denote the reduction of $G_n$ modulo $2$.
Buff, Floyd, Koch, and Parry made the surprising observation that the number of real roots of $G_n$ is equal to the number of irreducible factors of $\bar{G}_n$ for all $n$. We provide a bijective proof for this result by first providing explicit bijections between (a) the set of real roots of $G_n$ and the set $\bar{N}(n)$ of equivalence classes of primitive binary necklaces of length $n$ under the inversion map swapping $0$ and $1$; and (b) the set of irreducible factors of $G_n$ modulo 2 and the set $\tilde{N}^+(n)$ of binary necklaces which are either primitive of length $n$ with an even number of $1$'s or primitive of length $n/2$ with an odd number of $1$'s. We then provide an explicit bijection, closely related to Milnor and Thurston's kneading theory, between $\bar{N}(n)$ and $\tilde{N}^+(n)$. In addition, we provide explicit bijections between $\bar{N}(n)$, the set ${\rm CUP}(n)$ of cyclic unimodal permutations of $\{ 1,\ldots,n \}$, and the set $N^-(n)$ of primitive binary necklaces of length $n$ with an odd number of $1$'s.
\end{abstract}

\thanks{M.B. thanks Omid Amini, Jason Fulman, Sarah Koch, and Curt McMullen for helpful discussions. A.C, S.L., and M.Q. thank Alon Danai, David Fried, Joseph Kalarickal, Hahn Lheem, John Sim, Leonna Wang, and the PROMYS program for providing them with guidance throughout this project. M.B. was partially supported by NSF grant DMS-2154224 and a Simons Fellowship in Mathematics.}

\maketitle

\section{Introduction}
\label{sec:intro}

For $n \geq 1$, let $G_n(c)$ denote the $n^{\rm th}$ {\em Gleason polynomial}\footnote{It is conjectured that $G_n(c)$ irreducible over $\QQ$ for all $n \geq 1$, see \cite{BFKP} and the references therein.}, whose roots correspond to parameters $c$ such that the critical point $0$ is periodic of exact period $n$ under iteration of $z^2 + c$.
For example, 
$G_3(c) = \frac{(c^2 + c)^2 + c}{G_1(c)} = c^3 + 2c^2 + c + 1$ and 
\[
G_4(c) = \frac{((c^2 + c)^2 + c)^2 + c}{G_1(c)G_2(c)} = c^6 + 3c^5 + 3c^4 + 3c^3 + 2c^2 + 1.
\]

Let $\bar{G}_n(c)$ denote the reduction of $G_n$ modulo $2$, considered as an element of ${\mathbb F}_2[c]$.
Buff et.~al.~ made the surprising observation (cf.~\cite[Theorem 1.5]{BFKP}) that the number of real roots of $G_n$ is equal to the number of irreducible factors of $\bar{G}_n$ for all $n \geq 1$.
For example, $G_3(c) = c^3 + 2c^2 + c + 1$ has one real root and $\bar{G}_3(c) = c^3 + c + 1$ is irreducible over ${\mathbb F}_2$.
Similarly, $G_4(c)$ has two real roots and $\bar{G}_4(c) = (c^4 + c + 1)(c^2 + c + 1)$.

\medskip

It seems that this remarkable discovery has not yet been adequately ``explained'', in the sense that the theorem is proved in \cite{BFKP}
by counting both the number of real roots of $G_n$ and the number of irreducible factors modulo 2, and showing that the answer is the same in both cases. The common answer turns out to be $\gamma_n = \frac{1}{2n} \sum_{m | n, m {\rm \; odd}} \mu(m) 2^{\frac{n}{m}}$, where $\mu$ is the {\em M{\"o}bius function} from elementary number theory, characterized by the recursion 
\[
\sum_{d \mid n} \mu(d) = \left\{ 
\begin{array}{ll}
1 & {\rm if \;} n = 1 \\
0 & {\rm if \;} n > 1.
\end{array}
\right.
\]

The sequence $\gamma_n$ is labeled A000048 in the Online Encyclopedia of Integer Sequences. 

\begin{table}[ht] \caption{The sequence $\gamma_n$} 
 \centering    
 \begin{tabular}{c | c c c c c c c c c c}  
 \hline\hline                        
 $n$ & 1 & 2 & 3 & 4 & 5 & 6 & 7 & 8 & 9 & 10 \\
 \hline
$\gamma_n$ & 1 & 1 & 1 & 2 & 3 & 5 & 9 & 16 & 28 & 51 \\
 \hline     \end{tabular} 
  \end{table} 

\medskip

It is natural to wonder if there is an explicit bijection between the real roots of $G_n$ and the irreducible factors of $G_n$ modulo 2. 
The main result of this paper provides such an explicit bijection. In fact, we do more: we establish explicit bijections between numerous different sets which are all known to have cardinality $\gamma_n$, thus providing a better ``reason'' why they should have the same cardinality.

\medskip

A foothold into the problem is provided by a famous paper of Milnor and Thurston \cite{Milnor-Thurston}, 
which provides a combinatorial classification of the real roots of $G_n$ via \emph{kneading theory}.
In particular, the work of Milnor and Thurston can be used to show that the real roots of $G_n$ are canonically in bijection with (a) cyclic unimodal permutations of $\{ 1,\ldots, n \}$ (cf.~\cite[Chapter 5]{Diaconis-Graham}), (b) $I$-orbits of periodic cycles of length $n$ for the \emph{mod 1 doubling map} $D : \RR / \ZZ \to \RR / \ZZ$ sending $x$ to $2x \pmod{1}$, where $I : x \mapsto -x$, and (c) certain kinds of periodic cycles for the \emph{tent map} 
$T : [0,1/2] \to [0,1/2]$, defined as $T(x) = 2x$ if $0 \leq x \leq 1/4$ and $T(x) = 1-2x$ if $1/4 \leq x \leq 1/2$.
Bijection (a) is obtained by recording the way in which iteration of $z^2 + c$ permutes the elements of the critical orbit, and
bijection (b) can be obtained through the \emph{spider algorithm} \cite{Hubbard-Schleicher} (which is closely related to both the work of Milnor--Thurston \cite{Milnor-Thurston} and the Douady--Hubbard theory of \emph{external rays} \cite{Douady-Hubbard}). Bijection (c) is equivalent to (b) through elementary arguments.

\medskip

Moreover, if we define a \emph{necklace} to be an equivalence class of binary strings under cyclic rotation, and a \emph{necklace inversion class} to be an equivalence class of binary necklaces of length $n$ under the involution $\iota$ which swaps $0$ and $1$, it is easy to see that inversion classes of periodic cycles of length $n$ for the doubling map $D$ are in bijective correspondence with primitive necklace inversion classes of length $n$ (see Corollary~\ref{cor:D1toN3}).

\medskip

On the other hand, using an observation of Reutenauer \cite{Reutenauer}, it is not difficult (see Section~\ref{sec:NP} for details) to give a bijection between the irreducible factors of $\bar{G}_n$ and certain kinds of primitive necklaces of length $n$.
So our task is reduced to the purely combinatorial problem finding explicit bijections between such necklaces and primitive necklace inversion classes. This turns out to be closely connected to Milnor and Thurston's kneading theory, which we reformulate in a framework that appears to be new.\footnote{See \cite[Chapter II]{Collet-Eckmann} and \cite[Section 1.18]{Devaney} for related results and additional background on kneading theory.}

\medskip

It is interesting to note that although the problem of finding explicit bijections between certain types of necklaces and necklace inversion classes is a purely combinatorial one, making no reference to dynamical systems, our proofs make crucial use of ideas from dynamics.

\subsection{Statement of the main result}

In order to state our main theorem, we need a few definitions.

A permutation $\sigma$ of $\{ 1,\ldots, n \}$ is called {\em cyclic} if it consists of a single $n$-cycle, and {\em unimodal}\footnote{Some authors require unimodal permutations to increase and then decrease, but the two concepts are equivalent.
We note that $\sigma$ is unimodal if and only if $\sigma^{-1}$ is a {\em Gilbreath permutation} in the sense of \cite[Chapter 5]{Diaconis-Graham}.} 
if there is an integer $1 \leq m \leq n$ such that $\sigma$ is decreasing on $[1,m]$ and increasing on $[m,n]$.

A {\em binary necklace} of length $n$ is a binary string of length $n$, considered up to cyclic shifting. 
For example, $01100$, $11000$, and $10001$ all define the same binary necklace of length 5.

We denote the necklace associated to a binary string $s$ of length $n$ by $\langle s \rangle$. Formally, 
$\langle s \rangle$ is the orbit of $s$ under the natural action (by translation modulo $n$) of the group $\ZZ/n\ZZ$ on the set $\{ 0,1 \}^n$ of binary strings of length $n$.

A binary string $s$ is called {\em primitive} if no non-trivial cyclic shift of $s$ is equal to $s$. For example, $0101$ is not primitive, but $0011$ is. We say that a necklace is primitive if some (equivalently, every) representative string is primitive.

A polynomial over a field $F$ is called {\em centered} if the sum of its roots (in some algebraic closure of $F$) is equal to $0$, and {\em non-centered} otherwise.

\begin{thmA} \label{thm:MainThm}
For each positive integer $n$, there are explicit bijections between the following sets, all of which have cardinality $\gamma_n$:
\begin{enumerate}
\item[(M1)] Real numbers $c$ for which the critical orbit of $z^2 + c$ has exact period $n$, i.e., real roots of $G_n$. (These are known in the literature as {\em real hyperbolic centers}.)
\item[(M2)] Irreducible factors of $\bar{G}_n$.
\item[(D1)] The set $\bar D(n)$ of orbits, under the involution $I : x \mapsto -x$, of periodic cycles of primitive period $n$ for the mod 1 doubling map $D : \RR / \ZZ \to \RR / \ZZ$.
\item[(P1)] The set ${\rm CUP}(n)$ of cyclic unimodal permutations of $[n] := \{ 1,\ldots,n \}$.
\item[(N1)] The set $N^-(n)$ of primitive binary necklaces of length $n$ with an odd number of $1$'s.
\item[(N2)] The set $\tilde{N}^+(n)$ of binary necklaces which are either primitive of length $n$ with an even number of $1$'s or primitive of length $n/2$ with an odd number of $1$'s.
\item[(N3)] The set $\bar{N}(n)$ of equivalence classes of primitive binary necklaces of length $n$ under inversion (sending $0$'s to $1$'s and vice-versa).
\item[(I1)] The set $I^-(n)$ of non-centered irreducible polynomials of degree $n$ over ${\mathbb F}_2$.
\item[(I2)] The set $\tilde{I}^+(n)$ of centered irreducible polynomials of degree $n$ over ${\mathbb F}_2$, together with the set of non-centered irreducible polynomials of degree $n/2$ over ${\mathbb F}_2$.
\end{enumerate}
\end{thmA}

We briefly summarize which parts of Theorem~\ref{thm:MainThm} are new, and which parts are already known:

\begin{itemize}
\item As already mentioned, Milnor and Thurston \cite{Milnor-Thurston} gave an explicit bijection between (M1), (P1), and (D1); see \cite{Hubbard-Schleicher} for a simplified (but still difficult) approach based on the \emph{spider algorithm}.  
\item The authors of \cite[Theorem 1.5]{BFKP} show that the sets (M2) and (I2) coincide.
\item A bijection between (P1) and (N1) follows from the work of Weiss and Rogers \cite{Weiss-Rogers}; this is made explicit in \cite{Archer-Elizalde}, although we only learned of the latter reference after independently observing this fact. 
We describe the results of \cite{Weiss-Rogers}, using the terminology and notation from this paper, in Section~\ref{sec:Weiss-Rogers}.
\item Reutenauer's modification of a bijection due to Golomb can be used to give an explicit bijection between (N1) and (I1), and between (N2) and (I2); see Section~\ref{sec:NP} for details.
\end{itemize}

Our main contribution in this paper is the construction of explicit canonical bijections between (P1), (N2), and (N3). By the above remarks, this suffices to prove Theorem~\ref{thm:MainThm}. 

For readers whose primary goal to understand the bijection between (M1) and (M2),  
it suffices to read the proof of Theorem~\ref{thm:N2N3}, along with the background material in Sections~\ref{sec:NP} and \ref{sec:MT}.

\begin{remark}
We note that, even though all the sets appearing in Theorem~\ref{thm:MainThm} have the same cardinality, it was not necessarily expected that one could find explicit and/or natural bijections between many of them. Indeed, Weiss and Rogers write in \cite[p.704]{Weiss-Rogers}, ``There are other combinatorial settings in which the formula of the theorem occurs\ldots however we emphasize there are no natural bijections which relate the various applications, and the ubiquity of the formula merely reflects the usefulness of the M\"obius function in counting situations.'' In some sense, the present paper can be read as an extended counterexample to this claim by Weiss and Rogers. 
\end{remark}

\subsection{Necklace bijections} \label{sec:NecklaceBijections}

Here and elsewhere in the paper, we switch interchangeably between thinking of a non-primitive necklace in $\tilde N^+(n)$ as either being primitive of length $n/2$ or having length $n$ and consisting of two copies of such a primitive necklace. For example, we identify the primitive necklace $\langle 10 \rangle \in N^-(2)$ with the non-primitive necklace $\langle 1010 \rangle \in \tilde N^+(4)$.

\medskip

Our bijection between (N2) and (N3) is easy to describe.
Let $[s]$ denote the necklace inversion class corresponding to a binary string $s \in \{ 0,1 \}^n$.
Formally, $[s]$ is the orbit of $s$ under the action of $\ZZ/n\ZZ \times \ZZ/2\ZZ$ on $\{ 0,1 \}^n$, where $\ZZ/n\ZZ$ acts by rotation and $\ZZ/2\ZZ$ acts by inversion.
Thus, $\bar{N}(n)$ is the set of necklace inversion classes represented by \emph{primitive} binary strings of length $n$.

\medskip

Define the ``mod 2 partial sum'' function $\Xi: \{ 0,1 \}^n \to \{ 0,1 \}^n$ as follows:
 for $s=(s_1,s_2,\ldots,s_n)\in \{ 0,1 \}^n$, we let $\Xi(s)$ be the corresponding sequence of partial sums modulo 2, i.e.,
\[
\Xi(s_1,\ldots,s_n) = (s_1,s_1 + s_2,s_1 + s_2 + s_3,\ldots,s_1 + s_2 + \cdots + s_n),
\]
where all additions are performed modulo 2.

It is easy to see that the map $\Xi$ is a bijection, with inverse given by
\[
\Xi^{-1}(s_1,\ldots,s_n) = (s_1,s_1 + s_2,s_2 + s_3,\ldots,s_{n-1} + s_n).
\]

For example, we have $\Xi(010111) = 011010$ and $\Xi^{-1}(011010)=010111$.

\begin{thmA} \label{thm:N2N3}
The function $\psi^+ : \tilde N^{+}(n)\rightarrow \bar N  (n)$ sending $\langle s \rangle$ to $[\Xi(s)]$ is a bijection.
\end{thmA}

This result will be proved in \S\ref{sec:N2N3}.

\begin{example}
As an illustration of Theorem~\ref{thm:N2N3}, 
$$ \tilde N^+(6) = \{ \langle 000011 \rangle, \langle 000101 \rangle, \langle 001111 \rangle, \langle 010111 \rangle, \langle 001001 \rangle \}$$
and 
$$ \bar N(6) = \{ [000010], [000110], [001010], [011010], [001110] \},$$
and we have $\Xi(000011)=000010, \ \Xi(000101)=000110$, etc.
\end{example}

We now describe a bijection from (P1) to (N2).

\medskip

Suppose $\sigma \in {\rm CUP}(n)$ is a cyclic unimodal permutation of $[n]$. 
Let $m$ be the unique element of $[n]$ with $\sigma(m) = 1$. 
Define ${\rm Itin}(\sigma) \in \{ +, -, \star \}^n$ by setting ${\rm Itin}(\sigma)_n = \star$, and for $1 \leq i \leq n-1$ setting ${\rm Itin}(\sigma)_i = +$ if $\sigma^i(m) > m$ and ${\rm Itin}(\sigma)_i = -$ if $\sigma^i(m) < m$.

This gives a sequence ${\rm Itin}(\sigma) \in \{ +, -, \star \}^n$ which we call the \emph{itinerary} of $\sigma$.

We define $A(\sigma) \in \{ 0,1 \}^n$ by
replacing all occurrences of $+$ in ${\rm Itin}(\sigma)$ with $0$, all occurrences of $-$ with $1$, and $\star$ by the unique element of $\{ 0,1 \}$ such that the mod-2 sum of the digits of $s$ is zero. 

\begin{example}
\begin{enumerate}
    \item[]
    \item If $\sigma = (15432)$, we have $m=2$ and ${\rm Itin}(\sigma) = -+++\star$. Thus $A(\sigma) = 10001$.
    \item If $\sigma = (164253)$ then $m=3$ and ${\rm Itin}(\sigma) = -++-+\star$. Thus $A(\sigma) = 100100$.
\end{enumerate}
\end{example}

\begin{thmA} \label{thm:P1N2}
\begin{enumerate}
    \item[]
    \item For every $\sigma \in {\rm CUP}(n)$, the corresponding necklace ${\phi}(\sigma) := \langle A(\sigma)\rangle$ belongs to $\tilde{N}^+(n)$.
    \item The resulting map $\phi : {\rm CUP}(n) \to \tilde{N}^+(n)$ is a bijection.
\end{enumerate} 
\end{thmA}

The proof of Theorem~\ref{thm:P1N2} will provide an explicit inverse for $\phi$, involving a bijection $\lambda$ between (N3) and (P1) which we will describe in \S\ref{sec:N1N2}. 
Roughly speaking, the map $\lambda : \bar{N}(n) \to {\rm CUP}(n)$ is defined by applying a certain dynamical system (closely related to the tent map) to a necklace in $\bar N(n)$, lexicographically ordering the elements of the resulting periodic orbit, and reading off a cyclic permutation from this ordering.

\medskip

Specifically, the following theorem will be proved along with Theorem~\ref{thm:P1N2}:

\begin{thmA} \label{thm:N3P1}
\begin{enumerate}
    \item[]
    \item The map $\lambda$ is a bijection.
    \item The composition $\phi \circ \lambda \circ \psi^+ : \tilde{N}^+(n) \to \tilde{N}^+(n)$ is the identity map.
\end{enumerate} 
\end{thmA}

The picture to have in mind here is that $\phi$, $\lambda$, and $\psi^+$ form a ``cycle of bijections", by which we mean sets $X,Y,Z$ and maps $f : X \to Y, g: Y \to Z, h : Z \to X$ with the property that if you go around the circle starting at any of $X$, $Y$, or $Z$, you get the identity map (i.e., $h \circ g \circ f = {\rm id}_X$, $f \circ h \circ g = {\rm id}_Y$, and $g \circ f \circ h = {\rm id}_Z$).

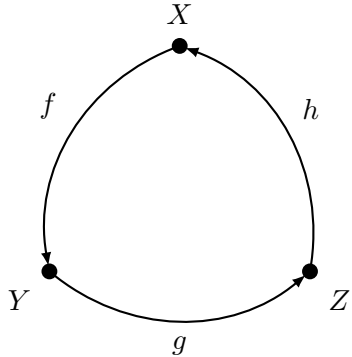
\begin{center}
\begin{tikzpicture}[scale=2, 
    every node/.style={font=\small},
    >=latex  
  ]
  \def\r{1}
  \coordinate (X) at (90:\r);    
  \coordinate (Y) at (210:\r);   
  \coordinate (Z) at (330:\r);   

  \fill (X) circle (1.5pt);
  \fill (Y) circle (1.5pt);
  \fill (Z) circle (1.5pt);

  \node[above=4pt] at (X) {$X$};
  \node[below left=4pt] at (Y) {$Y$};
  \node[below right=4pt] at (Z) {$Z$};

  \draw[->, thick, shorten >=2pt] (X) to[bend right=40] node[midway, above left=2pt] {$f$} (Y);
  \draw[->, thick, shorten >=2pt] (Y) to[bend right=40] node[midway, below=2pt] {$g$} (Z);
  \draw[->, thick, shorten >=2pt] (Z) to[bend right=40] node[midway, above right=2pt] {$h$} (X);
\end{tikzpicture}
\end{center}

In the process of proving Theorems~\ref{thm:P1N2} and \ref{thm:N3P1}, we will also establish a bijection between (N3) and (D1) (Corollary~\ref{cor:D1toN3} below).

\section{Background material} \label{sec:background}

\subsection{The Weiss--Rogers bijection} \label{sec:Weiss-Rogers}

In this section, we describe how the results of Weiss and Rogers \cite{Weiss-Rogers} can be used to give a concrete bijection between $N^-(n)$ and ${\rm CUP}(n)$.

\medskip

The Weiss--Rogers map $\Phi : N^-(n) \to {\rm CUP}(n)$ can be described in our language as follows:
 
\begin{enumerate}
\item[] 
\item For $s = (s_1, s_2,\ldots,s_n)$ with $\langle s \rangle \in N^-(n)$, 
we list the successive cyclic left-shifts of $s$, and to each one we apply the map $\Xi$ to obtain $\mu_1',\ldots,\mu_n'$, each ending in $1$:
$$ \mu_1' = \Xi(s_1, s_2,\ldots,s_n), \ \mu_2' = \Xi(s_2, s_3,\ldots,s_n, s_1),\ldots, \ \mu_n' = \Xi(s_n, s_1,\ldots, s_{n-1}).$$
\item Apply $\iota$ to each $\mu_i'$ to get $\nu_i = \iota(\mu_i')$, each ending in 0.
\item For $1 \leq i \leq n$ let $r_i$ denote the lexicographic ranking of $\nu_i$, i.e., one plus the number of $\nu_j$ which are lexicographically smaller than $\nu_i$.
\item Define $\Phi(\langle s \rangle)$ to be the cyclic permutation $\sigma := (r_1 r_2 \cdots r_n)$.
\end{enumerate}

\begin{example}
\begin{enumerate}
\item[] 
\item If $s=0001$ then  
\[
\begin{aligned}
\mu_1' &= \Xi(0001) = 0001 \Rightarrow \nu_1 = 1110 \\
\mu_2' &= \Xi(0010) = 0011 \Rightarrow \nu_2 = 1100 \\
\mu_3' &= \Xi(0100) = 0111 \Rightarrow \nu_3 = 1000 \\
\mu_4' &= \Xi(1000) = 1111 \Rightarrow \nu_4 = 0000. \\
\end{aligned}
\]

The lexicographic ranking of these strings is 4, 3, 2, 1, respectively, and so we get 
\[
\Phi(\langle 0001 \rangle) = (4 3 2 1) = (1 4 3 2).
\]

\item If $s=0111$ then 
\[
\begin{aligned}
\mu_1' &= \Xi(0111) = 0101 \Rightarrow \nu_1 = 1010 \\
\mu_2' &= \Xi(1110) = 1011 \Rightarrow \nu_2 = 0100 \\
\mu_3' &= \Xi(1101) = 1001 \Rightarrow \nu_3 = 0110 \\
\mu_4' &= \Xi(1011) = 1101 \Rightarrow \nu_4 = 0010. \\
\end{aligned}
\]

The lexicographic ranking of these strings is 4, 2, 3, 1, respectively, and so we get 
\[
\Phi(\langle 0111 \rangle) = (4 2 3 1) = (1 4 2 3).
\]

\end{enumerate}
\end{example}

The Weiss--Rogers map $\Psi : {\rm CUP}(n) \to N^-(n)$ can be described as follows:

\begin{enumerate}
    \item Start with a cyclic unimodal permutation $\sigma$ and let ${\rm Itin}(\sigma)$ be its itinerary, as defined in Section~\ref{sec:NecklaceBijections}.
    \item Replace $-$ by $1$, $+$ by $0$, and $\star$ by whatever is needed to make a string $s$ with an odd number of 1’s (compare with the definition of $\phi$ from Section~\ref{sec:NecklaceBijections}).
    \item Define $\Psi(\sigma)$ to be $\langle s \rangle$.
\end{enumerate}

\begin{example}
\begin{enumerate}
    \item[]
    \item Starting with $\sigma = (1 4 3 2)$ we get the itinerary $(-,+, +, \star)$, which gives the string $s = 1000$.
    \item Starting with $\sigma = (1 4 2 3)$, we get the itinerary $(-, +, -, \star)$, which gives the string $s=1011$.
\end{enumerate}
\end{example}

\begin{theorem}[Weiss--Rogers, \cite{Weiss-Rogers}]  \label{thm:N1P1}
The maps $\Phi: N^{-}(n)\rightarrow {\rm CUP}(n)$ and $\Psi : {\rm CUP}(n) \to N^-(n)$ are inverse to one another.
\end{theorem}

We provide a brief ``dictionary'' for translating between Theorem~\ref{thm:N1P1}, as we've formulated it here, and the main theorem of \cite{Weiss-Rogers}, which establishes a bijection between sets which Weiss and Rogers call $\Delta_n$ and $E_{-}(n)$. In a nutshell, their set $\Delta_n$ is exactly our ${\rm CUP}(n)$, and their $E_{-}(n)$ is canonically in bijection with our $N^-(n)$. We explain the latter bijection now.

\medskip

Let $E$ denote the set $\{ \pm 1 \}^{\NN}$ of sequences $\epsilon = (\epsilon_1,\epsilon_2,\ldots)$ with $\epsilon_i \in \{ -1, 1 \}$ for all $i$. We order $E$ lexicographically, with the convention that $-1 < 1$.
Following Weiss--Rogers, one constructs a dynamical system on $E$ by letting $F : E \to E$ be the {\em twisted shift operator} defined by
$F(\epsilon)_i = \epsilon_1 \cdot \epsilon_{i+1}$
for $i=1,2,\ldots$
In other words, $F$ acts as a left shift operator followed by multiplication with $\epsilon_1$. (This is closely related to the twisted shift operator $F$ described in Section~\ref{sec:symbolicdynamics} below.)

Write ${\rm Per}(n)$ for the set of elements of $E$ with $F^{\circ n}(\epsilon) = \epsilon$. A sequence $\epsilon \in {\rm Per}(n)$ is determined by its first $n$ coordinates, so one sometimes thinks of elements of ${\rm Per}(n)$ as $n$-tuples of the form $(\epsilon_1,\ldots,\epsilon_n)$, rather than as infinite strings.
If $C$ is a periodic cycle with period $n$, the $n^{\rm th}$ coordinate $\epsilon_n$ is independent of the choice of $\epsilon \in C$; Weiss and Rogers call this the {\em signature} of $C$ (or of $\epsilon$).

Let ${\rm Per}_{-}(n)$ (resp. ${\rm Per}_{+}(n)$) denote the set of elements of ${\rm Per}(n)$ with odd (resp. even) signature, and let $E_{-}(n)$ 
(resp. $E_{-}(n)$) denote the set of odd (resp. even) periodic cycles with primitive period $n$ for the iteration of $F$. 

Given a binary sequence $s \in \{ 0,1 \}^{\NN}$, we can define a corresponding sequence $\omega(s) \in E$ as follows.
Write $s = (a_1,a_2,\ldots)$ with $a_i \in \{ 0,1 \}$.
We define $\omega(s) = \epsilon$, where $\epsilon_1 = (-1)^{a_1}$ and $\epsilon_{i+1} = (-1)^{a_i + a_{i+1}}$ for $i = 1,2,\ldots$ 
(This is closely related to our cumulative mod 2 sum map $\Xi$.)
The importance of $\omega$ is that it semi-conjugates $F$ to the left-shift operator $L : \{ 0,1 \}^{\NN} \to \{ 0,1 \}^{\NN}$, i.e., 
$F \circ \omega = \omega \circ L$.

\begin{example}
The map $\omega$ takes $0111 \in N^-(n)$ to $(1,-1,1,-1) \in {\rm Per}_{-}(n)$. The $F$-orbit of this point is the 4-cycle 
\[
(1,-1,1,-1) \mapsto (-1,1,-1,-1) \mapsto (-1,1,1,-1) \mapsto (-1,-1,1,-1),
\]
which is the image under $\omega$ of the orbit of $0111$ under the cyclic left-shift operator:
\[
0111 \mapsto 1110 \mapsto 1101 \mapsto 1011.
\]
\end{example}

It is not difficult to prove that if $\langle s \rangle \in N^-(n)$, the cycle generated by $\epsilon := \omega(s)$ belongs to $E_{-}(n)$ and depends only on $\langle s \rangle$. Moreover, the induced map $\omega : N^-(n) \to E_{-}(n)$ is a bijection.

It follows from the definitions that if we identify $N^-(n)$ with $E_{-}(n)$ using the bijection $\omega$, the map from $E_{-}(n)$ to ${\rm CUP}(n)$ which Weiss and Rogers call $\Phi$ corresponds to the map $\Phi$ described above, and the map from ${\rm CUP}(n)$ to $E_{-}(n)$ which Weiss and Rogers call $\Psi$ corresponds to the map $\Psi$ described above. We leave the details to the interested reader.

\begin{remark}
The methods which Weiss and Rogers employ make crucial use of the fact that they're considering only cycles with \emph{odd} signature.\footnote{Indeed, Weiss and Rogers write on p.704 of their paper, ``Cycles of $F$ have a signature. This is what makes the other half of Lemma 3 (i.e., $\pi_\delta$ is injective) work, yet still remain somewhat mysterious.''}
In particular, their arguments do not seem to extend to the even signature case considered in this paper (we tried, and failed, to imitate their proofs in the context of $\tilde N^+(n)$). For this reason, we take a somewhat different approach below, although we were certainly influenced by the ideas in \cite{Weiss-Rogers}.
\end{remark}

\subsection{Necklaces and polynomials}
\label{sec:NP}

In \cite{Golomb}, Golomb gives a bijection between primitive binary necklaces of length $n$ and irreducible polynomials of degree $n$ over $\FF_2$.  
The bijection works as follows. 

First, by Galois theory, there is a canonical bijection between irreducible polynomials of degree at most $n$ over $\FF_2$ and ${\rm Gal}(\FF_{2^n}/\FF_2)$-orbits of elements of $\FF_{2^n}^\times$, with the degree of the polynomial corresponding to the size of the orbit.
Since ${\rm Gal}(\FF_{2^n}/\FF_2)$ is cyclic of order $n$, with the Frobenius automorphism $\varphi : z \mapsto z^2$ as a canonical generator, there is a non-canonical bijection between 
irreducible polynomials of degree $n$ over $\FF_2$ and orbits of size $n$ of the multiplication-by-2 map on $(\ZZ / (2^n - 1)\ZZ)$.
This bijection depends on the choice of an isomorphism $\FF_{2^n}^\times \cong (\ZZ / (2^n - 1)\ZZ)$, or equivalently, on the choice a primitive element $\alpha \in \FF_{2^n}$ generating the multiplicative group $\FF_{2^n}^\times$.
The irreducible polynomial associated to an orbit of the multiplication-by-2 map on $(\ZZ / (2^n - 1)\ZZ)$ containing the integer $k$ is the minimal polynomial of $\alpha^k$.

If we choose an element of $\{ 1,2,\ldots, 2^n - 1 \}$ representing a given class in $(\ZZ / (2^n - 1)\ZZ)$ and write it in binary, we obtain a binary string $s$ of length $n$, and the orbit of $s$ under the multiplication-by-2 map corresponds naturally to the necklace $\langle s \rangle$.
Under this correspondence, orbits of length $n$ correspond to primitive necklaces. This yields Golomb's non-canonical bijection between primitive necklaces of length $n$ and irreducible polynomials of degree $n$ over $\FF_2$.

\medskip

Unfortunately, Golomb's bijection does not, in general, respect traces --- in other words, it does not induce a correspondence between binary necklaces $\nu = \langle a_1 a_2 \cdots a_n \rangle$ with $\sum a_i$ equal to $a$ and polynomials for which the ``trace'' (i.e., the sum of the roots) is $a$. 
Since our goal is to match up $N^-(n)$ with $I^-(n)$ and $\tilde{N}^+(n)$ with $\tilde{I}^+(n)$, Golomb's correspondence is not very useful for our purposes. 
Fortunately, a simple variation of Golomb's idea due to Reutenauer \cite{Reutenauer} works perfectly in the present context.

\medskip

Instead of choosing a primitive element, we instead use the normal basis theorem to choose an element $\beta \in \FF_{2^n}^\times$ such that 
$\{ \beta, \varphi(\beta), \varphi^2(\beta),\ldots, \varphi^{n-1}(\beta) \}$ 
is a basis for $\FF_{2^n}$ as a vector space over $\FF_2$.
Having made such a choice, we can use the ordered basis 
\[
(\varphi^{n-1}(\beta), \ldots, \varphi^2(\beta), \varphi(\beta), \beta)
\]
to identify elements of $\FF_{2^n}$ with binary strings of length $n$. Like Golomb's original bijection, Reutenauer's correspondence also has the property that ${\rm Gal}(\FF_{2^n}/\FF_2)$-orbits of elements of $\FF_{2^n}^\times$ correspond to binary necklaces of length $n$, with orbits of size $n$ corresponding to primitive necklaces. 
But the sum of the elements in a given Galois orbit (which is the trace of the corresponding irreducible polynomial) now visibly corresponds to the mod 2 sum of the entries in the associated necklace.

\medskip

As a consequence of this discussion, we obtain:

\begin{theorem}
Fix an element $\beta \in \FF_{2^n}^\times$ such that the ${\rm Gal}(\FF_{2^n}/\FF_2)$-orbits of $\beta$ form a normal basis for $\FF_{2^n}$ over $\FF_2$.
Then there is an explicit bijection from primitive binary necklaces of length $d$ dividing $n$ to irreducible polynomials of degree $d$ over $\FF_2$.
Moreover, this bijection takes $N^-(n)$ to $I^-(n)$ and $\tilde{N}^+(n)$ to $\tilde{I}^+(n)$.
\end{theorem}

\begin{example}
Let $n=4$, let $\alpha \in \FF_{16}$ be a root of $x^4 + x + 1 \in \FF_2[x]$, and let $\beta = \alpha^3$. 
One can check that $\{ \beta,\beta^2,\beta^4,\beta^8 \}$ are linearly independent over $\FF_2$, i.e., they form a normal basis for $\FF_{16}$ over $\FF_2$. With respect to the ordered basis $(\beta^8,\beta^4,\beta^2,\beta)$, the string $0011$ corresponds to the element $\beta^2 + \beta \in \FF_{16}$, and the irreducible polynomial in $\tilde{I}^+(4)$ corresponding to the necklace $\langle 0011 \rangle \in \tilde{N}^+(4)$ is the minimal polynomial over $\FF_2$ of $\beta^2 + \beta$.
Using the relation $\alpha^4 = \alpha + 1$, we see that $\beta^2 + \beta = \alpha^6 + \alpha^3 = \alpha^2$, which is conjugate to $\alpha$. Therefore
the minimal polynomial of $\beta^2 + \beta$ is the same as the minimal polynomial of $\alpha$, namely $x^4 + x + 1$.

Similarly, the string $0101$ corresponds to the element $\beta^4 + \beta \in \FF_{16}$, 
and the irreducible polynomial in $\tilde{I}^+(4)$ corresponding to the necklace $\langle 0101 \rangle \in \tilde{N}^+(4)$ 
is minimal polynomial of $\beta^4 + \beta$.
One computes that $\beta^4 + \beta = \alpha^{12} + \alpha^3 = \alpha^2 + \alpha + 1$. Moreover, 
\[
(\alpha^2 + \alpha + 1)^2 = \alpha^4 + \alpha^2 + 1 = (\alpha + 1) + \alpha^2 + 1 = \alpha^2 + \alpha.
\]
It follows that the minimal polynomial of $\beta^4 + \beta$ is $x^2 + x + 1$.
\end{example}

\subsection{Periodic points of the doubling and tent maps} \label{sec:DandT}

Recall that the mod 1 doubling map $D : \RR / \ZZ \to \RR / \ZZ$ is defined by sending $x$ to $2x \pmod{1}$.
In terms of binary expansions, $D$ shifts the binary expansion of $x$ leftward, dropping the first digit,
i.e., if $x = 0.a_1a_2a_3\ldots$ then $D(x) = 0.a_2a_3a_4\ldots$
One deduces easily from this description the well-known fact that periodic points of $D$ correspond to real numbers $x \in [0,1)$ whose binary expansion is periodic, which means that $x$ is a rational number whose denominator has the form $2^n - 1$ for some integer $n \geq 1$. More precisely, $x \in (0,1)$ has primitive period $n \geq 2$ under iteration of $D$ iff $x = \frac{a}{2^n - 1}$ for some integer $a\geq 1$.
In particular, periodic \emph{cycles} of length $n \geq 2$ for $D$ are in bijection with primitive binary necklaces of length $n$.

\begin{example}
There are two periodic cycles for $D$ of length 3, namely 
$1/7 \mapsto 2/7 \mapsto 4/7$ and $3/7 \mapsto 6/7 \mapsto 5/7$.
In binary, these periodic cycles can be written as $.\overline{001} \mapsto .\overline{010} \mapsto .\overline{100}$ and 
$.\overline{011} \mapsto .\overline{110} \mapsto .\overline{101}$, respectively, corresponding to the necklaces $\langle 001 \rangle$ and $\langle 011 \rangle$.

There are three periodic cycles for $D$ of length 4, namely
$1/15 \mapsto 2/15 \mapsto 4/15 \mapsto 8/15$, $7/15 \mapsto 14/15 \mapsto 13/15 \mapsto 11/15$, and 
$1/5 \mapsto 2/5 \mapsto 4/5 \mapsto 3/5$, which can be rewritten as $3/15 \mapsto 6/15 \mapsto 12/15 \mapsto 9/15$.
In binary, these periodic cycles can be written as $.\overline{0001} \mapsto .\overline{0010} \mapsto .\overline{0100} \mapsto .\overline{1000}$, $.\overline{0111} \mapsto .\overline{1110} \mapsto .\overline{1101} \mapsto .\overline{1011}$, and $.\overline{0011} \mapsto .\overline{0110} \mapsto .\overline{1100} \mapsto .\overline{1001}$, respectively, corresponding to the necklaces $\langle 0001 \rangle$, $\langle 0111 \rangle$, and $\langle 0011 \rangle$.
\end{example}

We now recall the relationship (which is well-known in the theory of one-dimensional dynamical systems) between periodic cycles for $D$ and periodic cycles for the tent map $T : [0,1/2] \to [0,1/2]$ defined by 
\[
T(x) = 
\begin{cases}
2x & \text{if $0 \leq x \leq 1/4$} \\
1-2x & \text{if $1/4 \leq x \leq 1/2$}. \\ 
\end{cases}
\]

\begin{figure}[h!]
\centering
\begin{tikzpicture}
  \begin{axis}[
    axis lines=middle,
    xmin=0, xmax=0.55,
    ymin=0, ymax=0.55,
    xtick={0, 0.25, 0.5},
    xticklabels={$0$, $\frac{1}{4}$, $\frac{1}{2}$},
    ytick={0, 0.25, 0.5},
    yticklabels={$0$, $\frac{1}{4}$, $\frac{1}{2}$},
    xlabel={$x$},
    ylabel={$T(x)$},
    domain=0:0.5,
    samples=200,
    width=10cm,
    height=8cm,
    grid=both,
    thick
  ]
    \addplot[blue, thick, domain=0:0.25] {2*x};
    \addplot[blue, thick, domain=0.25:0.5] {1 - 2*x};
  \end{axis}
\end{tikzpicture}
\caption{The tent map $T$}
\end{figure}
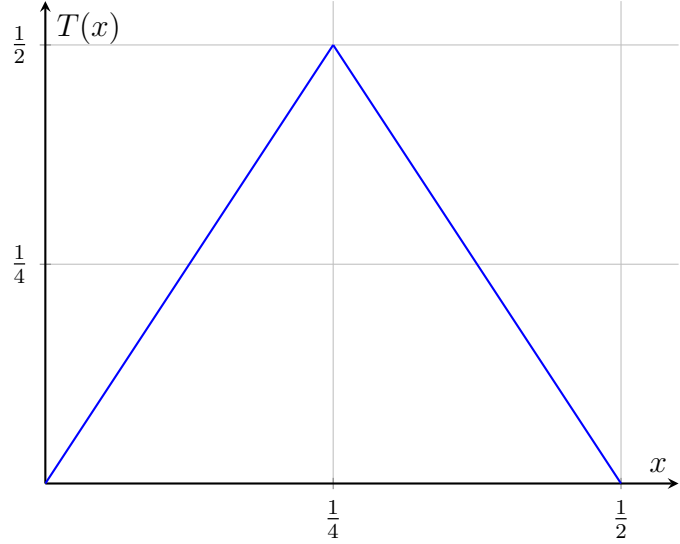

The basic fact is that every periodic orbit $\tilde{C}$ of $D$ of primitive period $n$ gives rise to a periodic orbit $C$ of $T$ with primitive period either $n/2$ or $n$, depending on whether or not $\tilde{C}$ is preserved by the involution $I : \RR / \ZZ \to \RR / \ZZ$ defined by $I(x)=-x$.
And, conversely, every periodic orbit of the tent map $T$ of primitive period $m$ can be lifted to a periodic orbit of the doubling map $D$ of primitive period either $m$ or $2m$. 

More precisely, periodic orbits $C$ of $T$ are in bijective correspondence with $I$-orbits of periodic orbits $\tilde{C}$ of $D$, and these correspond bijectively to rotation-inversion equivalence classes of binary necklaces.

\begin{example}
There is a single periodic cycle for $T$ of length 3, namely 
$1/7 \mapsto 2/7 \mapsto 3/7$.
This periodic cycle corresponds to the two periodic cycles $1/7 \mapsto 2/7 \mapsto 4/7$ and $3/7 \mapsto 6/7 \mapsto 5/7$ of $D$, which are swapped by the involution $I$.
The corresponding rotation-inversion equivalence class is $[001]$.

There are two corresponding periodic cycles for $D$, one of primitive period 4 and one of primitive period 2, namely $1/15 \mapsto 2/15 \mapsto 4/15 \mapsto 8/15$ (whose image under $I$ is $7/15 \mapsto 14/15 \mapsto 13/15 \mapsto 11/15$) and $1/5 \mapsto 2/5$ (which corresponds to the periodic cycle
$1/5 \mapsto 2/5 \mapsto 4/5 \mapsto 3/5$ for $D$, which is fixed by $I$).
The corresponding rotation-inversion equivalence class is $[01]$.
\end{example}

The close relationship between periodic cycles of $D$ and $T$ is explained by the semi-conjugacy from $D$ to $T$ given for $x \in [0,1)$ by $\pi(x) = \min(x,1-x)$, which is often referred to as the \emph{folding map}. The fact that $\pi : \RR / \ZZ \to [0,1/2]$  is a semi-conjugacy simply means that the following diagram commutes:
\[
\begin{tikzpicture}[node distance=3cm, auto]
  \node (A) at (0,2) {$[0,1)$};
  \node (B) at (3,2) {$[0,1)$};
  \node (C) at (0,0) {$[0,1/2]$};
  \node (D) at (3,0) {$[0,1/2]$};

  \draw[->] (A) to node {$D$} (B);
  \draw[->] (A) to node[left] {$\pi$} (C);
  \draw[->] (B) to node[right] {$\pi$} (D);
  \draw[->] (C) to node {$T$} (D);
\end{tikzpicture}
\]

\medskip

We now summarize, in a more detailed fashion, the precise correspondence between periodic cycles for $D$ and $T$:

\begin{proposition} \label{prop:DvsT}
Let $\tilde{C} = x_1 \mapsto x_2 \mapsto \cdots x_n$  be a periodic cycle of primitive period $n$ for $D$ with all $x_i \in [0,1)$, and let $y_i = \pi(x_i)$. 
\begin{enumerate}
     \item If $I(\tilde C) \neq \tilde{C}$ then $\pi_*(\tilde C) := y_1 \mapsto y_2 \mapsto \cdots \mapsto y_n$ is a periodic cycle of primitive period $n$ for $T$.
    \item If $I(\tilde C) = \tilde{C}$, then $n$ is even and $\pi_*(\tilde C) := y_1 \mapsto y_2 \mapsto \cdots \mapsto y_{n/2}$ is a periodic cycle of primitive period $n/2$ for $T$. 
\end{enumerate}

Conversely, if $C = y_1 \mapsto y_2 \mapsto \cdots y_m$ is a periodic cycle of primitive period $m$ for $T$, then either $C$ is also a periodic cycle of primitive period $m$ for $D$ or there is a unique periodic cycle 
$\tilde{C} = x_1 \mapsto x_2 \mapsto \cdots x_{2m}$ of primitive period $2m$ for $D$ such that $\pi_*(\tilde C)=C$.
\end{proposition}

Motivated by Proposition~\ref{prop:DvsT}, we define $\tilde T(n) := T_1(n) \ \dot\cup \ T_2(n)$, where
$T_2(n)$ consists of the periodic cycles $C$ of primitive period $n$ for $T$ which are also periodic cycles of primitive period $n$ for $D$ and $T_1(n)$ consists of the periodic cycles $C$ of primitive period $n/2$ for $T$ such that $\pi_*(\tilde C)=C$ for some periodic cycle $\tilde{C}$ of primitive period $n$ for $D$.

We define $\bar D(n)$ to be the set of $I$-orbits of periodic cycles $C$ of primitive period $n$ for $D$, and we write $\bar D(n) := \bar D_1(n) \ \dot\cup \ \bar D_2(n)$, where $\bar D_1(n)$ (resp. $\bar D_2(n)$) consists of all $I$-orbits represented by a periodic cycle $C$ with $I(C)=C$ (resp. $I(C) \neq C$).

With this terminology, we have:

\begin{corollary} \label{cor:D1toD2}
There is a bijection from $\bar D(n)$ to $\tilde T(n)$ which sends the $I$-orbit of $C$ to $\pi_*(C)$. This bijection sends $\bar D_1(n)$ to $\tilde T_1(n)$ and $\bar D_2(n)$ to $\tilde T_2(n)$.
\end{corollary}

\subsection{Symbolic dynamics of the doubling and tent maps} \label{sec:symbolicdynamics}

We fix the fundamental domain $[0,1)$ for $\RR / \ZZ$.
Recall that, on the level of binary expansions, $D : [0,1) \to [0,1)$ shifts the binary expansion of $x$ leftward, dropping the first digit.
It follows that if $P(n)$ denotes the set of periodic points for $D$ in $[0,1)$ of primitive period $n$, $\beta : P(n) \to \{ 0,1 \}^n$ is the map which sends $.\overline{a_1 a_2 \cdots a_n}$ to $a_1a_2\cdots a_n$, and $L : \{ 0,1 \}^n \to \{ 0,1 \}^n$ denotes the cyclic left-shift operator, then 
$\beta$ semi-conjugates $P(n)$ to $\{ 0,1 \}^n$, i.e.,
the following diagram commutes:
\[
\begin{tikzpicture}[node distance=3cm, auto]
  \node (A) at (0,2) {$P(n)$};
  \node (B) at (3,2) {$P(n)$};
  \node (C) at (0,0) {$\{ 0,1 \}^n$};
  \node (D) at (3,0) {$\{ 0,1 \}^n$};

  \draw[->] (A) to node {$D$} (B);
  \draw[->] (A) to node[left] {$\beta$} (C);
  \draw[->] (B) to node[right] {$\beta$} (D);
  \draw[->] (C) to node {$L$} (D);
\end{tikzpicture}
\]

This a prototypical example of \emph{symbolic dynamics}, where we encode the behavior of a dynamical system in terms of iterating some operator on strings.

\medskip

On the level of binary expansions, the involution $I : \RR / \ZZ \to \RR / \ZZ$ defined by $I(x)=-x$
corresponds to the \emph{inversion operator} $\iota$ taking a binary string to its complement, i.e., $\iota(0)=1$ and $\iota(1)=0$ and we extend this to strings $s = (s_1, s_2 \ldots s_n)$ by setting $\iota(s) = (\iota(s_1), \iota(s_2) \ldots \iota(s_n))$. 
In other words, if $x = 0.a_1a_2a_3\ldots$ then $I(x) = 0.\iota(a_1)\iota(a_2)\iota(a_3)\ldots$
We easily deduce from this:

\begin{corollary} \label{cor:D1toN3}
Given an element $\alpha \in \bar D(n)$, let $\tilde C$ be a periodic cycle for $D$ whose involution class is $\alpha$, and let $y$ be an element of the corresponding periodic cycle $\pi_*(\tilde C)$ for $T$. Write $y$ in binary as $y=.\overline{b_1 b_2 \cdots b_m}$. Then the map $\tau : \bar D(n) \to \bar N(n)$ defined by $\tau(\alpha)=[b_1 b_2 \cdots b_m]$ is a well-defined bijection.
\end{corollary}

Under the bijection from Corollary~\ref{cor:D1toN3}, $D_i(n)$ corresponds to $\bar N_i(n)$ for $i=1,2$, where $\bar N_1(n)$ is the set of all necklace inversion classes $[s]$ such that $\iota(s)$ is a cyclic shift of $s$ and
$\bar N_2(n)=\bar N(n)\setminus \bar N_1(n)$. 

\begin{example}
\begin{enumerate}
    \item[]
    \item We have $\iota(001110)=110001$, which is a cyclic shift of $001110$, so $[001110] \in \bar N_1(6)$. 
    \item On the other hand, $\iota(010000)=101111$ is not a cyclic shift of $010000$, so $010000 \in \bar N_2(6)$.
\end{enumerate}
\end{example}

\begin{remark} \label{rem:primitive-and-reflexive}
We will prove in Lemma~\ref{lem:primitive-and-reflexive} below that a necklace inversion class $[s]$ belongs to $\bar N_1(n)$ iff $s = s' \iota(s')$ for a primitive string $s'$ of length $n/2$.
Equivalently, the periodic cycle $\tilde C \in D(n)$ generated by a point $x = .\overline{a_1 a_2 \cdots a_n} \in [0,1)$ of primitive period $n$ for $D$ belongs to $D_1(n)$ (i.e., is fixed by $I$) iff $n$ is even and 
\[
x = .\overline{a_1 a_2 \cdots a_{n/2} b_1 b_2 \cdots b_{n/2}}
\]
with $b_i = \iota(a_i)$ for all $1 \leq i \leq n/2$.
For example, $x=.\overline{0110} = 2/5$ belongs to $D_1(4)$.
\end{remark}

In order to describe the tent map in terms of symbolic dynamics, it will be useful for our purposes to work instead with a modified tent map.
More precisely, let $T' : [1/2,1] \to [1/2,1]$ be the conjugate of the tent map $T$ by $I$, defined so that the diagram
\[
\begin{tikzpicture}[node distance=3cm, auto]
\node (A) at (0,2) {$[0,1/2]$};
\node (B) at (3,2) {$[0,1/2]$};
\node (C) at (0,0) {$[1/2,1]$};
\node (D) at (3,0) {$[1/2,1]$};

\draw[->] (A) to node {$T$} (B);
\draw[->] (A) to node[left] {$I$} (C);
\draw[->] (B) to node[right] {$I$} (D);
\draw[->] (C) to node {$T'$} (D);
\end{tikzpicture}
\]
commutes. 

We define the \emph{twisted shift operator} $F : \{0,1 \}^n \to \{ 0,1 \}^n$ by $F(s) = L(s)$ if $L(s)_1 = 1$ 
and $F(s) = \iota(L(s))$ if $L(s)_1 = 0$.
For example, $F(10010) = 11010$.

\medskip

Define $\pi' : [0,1] \to [1/2,1]$ by $\pi'(x)=\max\{ x, 1-x \}$, and define $\gamma : P_n \to \{ 0, 1 \}^n$ as $\gamma = \beta \circ \pi'$, i.e.,
\[
\gamma(x) = 
\begin{cases}
\beta(x) & \text{if $x > 1/2$} \\
\iota(\beta(x)) & \text{if $x < 1/2$}. \\
\end{cases}
\]

Then the diagram 
\[
\begin{tikzpicture}[node distance=3cm, auto]
  \node (A) at (0,4) {$P_n$};
  \node (B) at (5,4) {$P_n$};
  \node (C) at (0,2) {$T_n'$};
  \node (D) at (5,2) {$T_n'$};
  \node (E) at (0,0) {$\{ 0,1 \}^n$};
  \node (F) at (5,0) {$\{ 0,1 \}^n$};

  \draw[->] (A) to node {$D$} (B);
  \draw[->] (A) to node[left] {$\pi'$} (C);
  \draw[->] (B) to node[right] {$\pi'$} (D);
  \draw[->] (C) to node {$T'$} (D);
  \draw[->] (C) to node[left] {$\beta$} (E);
  \draw[->] (D) to node[right] {$\beta$} (F);
  \draw[->] (E) to node {$F$} (F);
\end{tikzpicture}
\]
commutes, where $T_n'$ is the image under $\pi'$ of $P_n$.

\subsection{Extended symbolic dynamics} \label{sec:extendedsymbolicdynamics}

Unfortunately, the symbolic description of the modified tent map $T'$ given in Section~\ref{sec:symbolicdynamics} will not be wholly adequate for our purposes. The issue is that while the $D$-orbit of every $x \in P_n$ has length $n$, 
the $T'$-orbit of certain points $x \in T_n'$ (resp.~the $F$-orbit of certain primitive strings $s \in \{ 0,1 \}^n$) only has length $n/2$. This will present a technical challenge for us later on, since our goal is to use the symbolic dynamics of $T'$, together with the natural lexicographic ordering on strings, to associate a cyclic unimodal permutation of $\{ 1, \ldots, n \}$ to each element of $T_n'$.
For this, we require an orbit of size $n$, rather than $n/2$.

\medskip

With this in mind, we define $\widetilde{[1/2,1]} := [1/2,1] \times \{ +, - \}$, and for notational convenience we write elements of $\widetilde{[1/2,1]}$ as $x^+$ or $x^-$, where $x \in [1/2,1]$. 
By abuse of terminology, we write $\iota$ for the involution on $\{ +, - \}$ which swaps $+$ and $-$.
We order $\widetilde{[1/2,1]}$ ``lexicographically'' by declaring that $x^a < y^b$ iff $x < y$ or $x=y$ and $a = -, b= +$.
We define the \emph{extended modified tent map} 
$\widetilde{T'} : \widetilde{[1/2,1]} \to \widetilde{[1/2,1]}$ by
\[
\widetilde{T'}(x^b) = 
\begin{cases}
T'(x)^b & \text{if $x > 3/4$} \\
T'(x)^{\iota(b)} & \text{if $x \leq 3/4$}. \\
\end{cases}
\]

\begin{example}
\begin{enumerate}
    \item[]
    \item The $\widetilde{T'}$-orbit of $3/5^-$ is
\[
3/5^- \mapsto 4/5^+ \mapsto 3/5^+ \mapsto 4/5^-.
\]
    \item The $\widetilde{T'}$-orbit of $35/63^-$ is
\[
35/63^- \mapsto 56/63^+ \mapsto 49/63^+ \mapsto 35/63^+ \mapsto 56/63^- \mapsto 49/63^-.
\]
\end{enumerate}
\end{example}

\medskip

For the symbolic counterpart of $\widetilde{T'}$, define an \emph{extended binary string of length $n$} to be an element of 
$\widetilde{\{ 0,1 \}^n} := \{ 0,1 \}^n \times \{ +, - \}$. For notational convenience, we write elements of $\widetilde{\{ 0,1 \}^n}$ as $s^+$ or $s^-$, where $s \in \{ 0,1 \}^n$. 
We order $\widetilde{\{ 0,1 \}^n}$ ``lexicographically'' by declaring that $s^a < t^b$ iff $s < t$ or $s=t$ and $a = -, b= +$.

We define the \emph{extended twisted shift operator} $\tilde{F}$ on $\widetilde{\{ 0,1 \}^n}$ by $\tilde{F}(s^b)=F(s)^b$ if $F(s)_n = 1$ and 
$\tilde{F}(s^b)=F(s)^{\iota(b)}$ if $F(s)_n = 0$.

\begin{example}
\begin{enumerate}
    \item[]
    \item The $\tilde{F}$-orbit of $1001^-$ is
\[
1001^- \mapsto 1100^+ \mapsto 1001^+ \mapsto 1100^-.
\]
    \item The $\tilde{F}$-orbit of $100011^-$ is
\[
100011^- \mapsto 111000^+ \mapsto 110001^+ \mapsto 100011^+ \mapsto 111000^- \mapsto 110001^-.
\]
\end{enumerate}
\end{example}

Define 
$\widetilde{\pi'} : [0,1] \to [1/2,1] \times \{ +, - \}$ by
\[
\widetilde{\pi'}(x) = 
\begin{cases}
x^+ & \text{if $x > 1/2$} \\
(1-x)^- &  \text{if $x \leq 1/2$}. \\
\end{cases}
\]

Define $\tilde\beta : T'_n \to \{ 0,1 \}^n \times \{ +, - \}$ by $\tilde\beta(x^b) = \beta(x)^b$.

Finally, define $\tilde \gamma : P(n) \to \{ 0, 1 \}^n \times \{ +, - \}$ as $\tilde \gamma = \tilde \beta \circ \widetilde{\pi'}$, i.e.,
\[
\tilde \gamma(x) = 
\begin{cases}
\beta(x)^+ & \text{if $x > 1/2$} \\
\iota(\beta(x))^- & \text{if $x \leq 1/2$}. \\
\end{cases}
\]

Then one verifies easily that the diagram 
\[
\begin{tikzpicture}[node distance=3cm, auto]
  \node (A) at (0,4) {$P_n$};
  \node (B) at (5,4) {$P_n$};
  \node (C) at (0,2) {$\widetilde{T_n'}$};
  \node (D) at (5,2) {$\widetilde{T_n'}$};
  \node (E) at (0,0) {$\widetilde{\{ 0,1 \}^n}$};
  \node (F) at (5,0) {$\widetilde{\{ 0,1 \}^n}$};

  \draw[->] (A) to node {$D$} (B);
  \draw[->] (A) to node[left] {$\widetilde{\pi'}$} (C);
  \draw[->] (B) to node[right] {$\widetilde{\pi'}$} (D);
  \draw[->] (C) to node {$\widetilde{T'}$} (D);
  \draw[->] (C) to node[left] {$\tilde\beta$} (E);
  \draw[->] (D) to node[right] {$\tilde\beta$} (F);
  \draw[->] (E) to node {$\tilde{F}$} (F);
\end{tikzpicture}
\]
commutes. 

\medskip

The primary advantage of this ``extended'' formalism is:

\begin{lemma} \label{lem:extendedformalism}
Suppose $s \in \{ 0,1 \}^n$ begins with $1$. If $[s] \in \bar N_2(n)$, define $\mu_1 = s$ and 
$\mu_i = F(\mu_{i-1})$ for $2 \leq i \leq n$. If $[s] \in \bar N_1(n)$, define $\mu_1 = s^-$ and $\mu_i = \tilde{F}(\mu_{i-1})$ for $2 \leq i \leq n$. Then the extended binary strings $\mu_1,\ldots,\mu_n$ are all distinct. 
\end{lemma}

To prove Lemma~\ref{lem:extendedformalism} in the case where $[s] \in \bar N_1(n)$, we require the following:

\begin{definition}
A binary string whose inversion is a cyclic rotation of itself is called \textit{reflexive}. Furthermore, given a binary string of length $n$, we call the string \textit{$k$-alternating} if it can be written as the concatenation of $k$ strings $s,\iota(s),s,\iota(s),\ldots,$ each with length $n/k.$
\end{definition}

For example, $s = 001110$ is reflexive, since $\iota(s) = 110001 = L^3(s)$.
And $s$ is also 2-alternating, since it is the concatenation of 001 and its inversion 110.
 
\begin{lemma} \label{lem:primitive-and-reflexive}
If $s \in \{ 0,1 \}^n$ is primitive and reflexive, then $s$ must be 2-alternating.
\end{lemma} 

\begin{proof}
Suppose $s = (s_0, s_1, \ldots, s_{n-1})$ is reflexive. 
Note that $n$ must be even, since otherwise $s$ and $\iota(s)$ cannot have the same number of 1's, and therefore cannot be rotations of one another.

By definition, there exists a $k$ with $1 \leq k \leq n$ such that
$L^k(s)=\iota(s)$, i.e.,
\[ 
(s_k, s_{1 + k} , \ldots, s_{(n-1) + k}) = (\iota(s_0),\iota(s_1), \ldots, \iota(s_{n-1})),
\]
where addition of indices is performed modulo $n$.

Therefore $\iota(s_j) = s_{j + k}$ for all $0 \leq j \leq n-1$,
and replacing $j$ by $j+k$ gives
$\iota(s_{j + k}) = s_{j + 2k}$.

Since $\iota$ is an involution, it follows that $s_j = s_{j + 2k}$ for all $0 \leq j \leq n-1$, i.e., $L^{2k}(s)=s$.
Since we have assumed $s$ to be primitive, it follows that $2k \equiv 0 \pmod n$.
Since $n$ is even, this implies $k \equiv 0 \pmod{n/2}$. As $1 \leq k \leq n$, we must have $k=n/2$.
In other words,
\[
(\iota(s_0), \iota(s_1), \ldots, \iota(s_{n/2 - 1})) = (s_{n/2}, s_{1 + n/2}, \ldots, s_{n - 1}),
\]
i.e., $s$ is 2-alternating.
\end{proof}

We can now give the promised proof of Lemma~\ref{lem:extendedformalism}.

\begin{proof}[Proof of Lemma~\ref{lem:extendedformalism}]
    When $[s] \in N_2(n)$, iterating $F$ clearly yields distinct representatives due to primitiveness. If $[s] \in \bar N_1(n)$, 
    the only non-trivial thing we need to check is that $\tilde{F}^{n/2}(s^-) = s^+$. 
    Note that when we apply $\tilde{F}$ to some extended string $s^{b}$, the sign changes iff the first two digits of $s$ are distinct. 
    It follows that the successive comparisons $s_1 \stackrel{?}{=} s_2 ,s_2 \stackrel{?}{=} s_3,..., s_{n/2} \stackrel{?}{=} s_{n/2+1}$ dictate whether the sign flips on that iteration of $\tilde{F}$.
    Since Lemma~\ref{lem:primitive-and-reflexive} proves that $s$ is 2-alternating for $[s] \in \bar N_1(n)$, we must have $s_1 \neq s_{n/2+1}$, which means that the sign flips an odd number of times over the course of $n/2$ iterations of $\tilde{F}$, as desired.
\end{proof}

We also have an ``extended'' version of the fact that $T'$ is unimodal (more precisely, $T'$ is decreasing on $(1/2,3/4)$ and increasing on $(3/4,1)$):

\begin{lemma}\label{lem:extendedunimodal}
If $x^a < y^b$ are both in $(1/2^+,3/4^-)$ then $\widetilde{T'}(x^a) > \widetilde{T'}(y^b)$,
and if $x^a < y^b$ are both in $(3/4^+,1^-)$ then $\widetilde{T'}(x^a) < \widetilde{T'}(y^b)$.
\end{lemma}

\begin{proof}
    Since $ T' =  I \circ  T \circ I$, it follows that $T'$ has the explicit description
    \[
T'(x) = 
\begin{cases}
2(1-x) & \text{if $1/2 \leq x \leq 3/4$} \\
2x-1 & \text{if $3/4 \leq x \leq 1$}. \\ 
\end{cases}
\]
In particular, $T'$ is decreasing on $(1/2,3/4)$ and increasing on $(3/4, 1)$.

Suppose first that $x^a < y^b$ are both in $(1/2^+,3/4^-)$. 
If $x^a < y^b$, either $x < y$ or both $x = y$ and $a < b$. In the first case, $T'$ being decreasing on $(1/2,3/4)$ directly implies $T'(x) > T'(y)$ and thus $\widetilde{T'}(x^a) > \widetilde{T'}(y^b)$. 
In the second case, $\widetilde{T'}(x^a) = T'(x)^a < \widetilde{T'}(y^a) = T'(x)^b$. 

If $x^a < y^b$ are both in $(3/4^+,1^-)$, the conclusion follows from a similar argument using the fact that $T'$ is increasing on $(3/4, 1)$.  
\end{proof}

\section{The bijection between (N2) and (N3)}
\label{sec:N2N3}

Our goal in this section is to prove Theorem~\ref{thm:N2N3}.

\medskip

We recall the definition of $\psi^+ : \tilde N^{+}(n)\rightarrow \bar N  (n)$ from the introduction: we have
$\psi^+(\langle s \rangle) = [\Xi(s)]$, where 
\[
\Xi(s_1,\ldots,s_n) = (s_1,s_1 + s_2,\ldots,s_1 + s_2 + \cdots + s_n)
\]
and all additions are performed modulo 2.

\medskip

We first divide each of $N^+(n)$ and $\bar N(n)$ into two disjoint subsets, which will end up corresponding to one another under the bijection $\psi^+$, and we analyze the two cases separately.

\medskip

First of all, by definition, $\tilde N^+(n) = N^+(n)$ if $n$ is odd and $\tilde N^+(n) = N^+(n)\ \dot\cup\  N^-(n/2)$ when $n$ is even, where $N^+(n)$ denotes the set of primitive binary necklaces of length $n$ with an even number of $1$'s.

\medskip

Next, recall from Section~\ref{sec:DandT} that $\bar N(n) = \bar N_1(n)\ \dot\cup\ \bar N_2(n)$, where
$\bar N_1(n)$ is the set of all necklace inversion classes $[s]$ such that $\iota(s)$ is a cyclic shift of $s$.

\begin{example}
The sets $N^+(6)$, $N^-(3)$, $N_1(6)$, and $N_2(6)$ are as follows:
\[
\begin{aligned}
\tilde N^+(6) &=  N^+(6)\ \dot\cup\  N^-(3) &= \{000011, 000101, 001111, 010111 \} \ \dot\cup\ \{ 001001\}. \\
\bar N(6) &= \bar N_2(6)\ \dot\cup\ \bar N_1(6) &=\{010000, 010100, 011000, 010110 \} \ \dot\cup\ \{ 011100\}. \\
\end{aligned}
\]
\end{example}

\subsection{$\psi^{+}$ is well-defined}

In order to check that $\psi^{+}$ is well-defined, we must verify the following:

\begin{enumerate}
    \item Every representative of $\langle s \rangle $ maps to a necklace in the same equivalence class.
    \item For $\langle s \rangle \in N^+(n)$, we have $\psi^{+}(\langle s \rangle )\in\bar N_2(n)$, and for $\langle s \rangle \in N^-(n/2)$, we have $\psi^{+}(\langle s \rangle )\in\bar N_1(n)$. 
\end{enumerate}

 For (1), we prove that if $\langle s_1 \rangle = \langle s_2 \rangle$, then $[\Xi(s_1)] = [\Xi(s_2)]$, which means that we can obtain $\Xi(s_2)$ by cyclically shifting $\Xi(s_1)$ and then inverting if necessary. 

\medskip

 \textbf{Case 1: }$s_1$ begins with 0. 

\medskip

 Let $s_1 = 0a_2a_3\cdots a_n \in \tilde N^+(n)$ with $\Xi(s_1) = 0b_2\cdots b_n$.
 It suffices to consider the case where $s_2 = L(s_1)$.
 In this case, $s_2 = L(s_1) = a_2a_3\cdots a_n0$,  $\Xi(s_2)_i = a_2 + \dots + a_{i+1} = b_{i+1}$ for $0\le i<n$ and (setting $a_1 = 0$) $\Xi(s_2)_n = a_1 + a_2 + \dots + a_n = 0$ (since $s_2$ has an even number of 1's). This means that $\Xi(s_2) = b_2 b_3 \cdots b_n 0 = L(\Xi(s_1))$.

\begin{center}
\begin{tikzcd}
0a_2a_3\cdots a_n \arrow[d, "\Xi"] \arrow[rr, "\mathrm{Rotate}"] &  & a_2a_3\cdots a_n 0 \arrow[d, "\Xi"] \\
0b_2b_3\cdots b_n \arrow[rr, "\mathrm{Rotate}"]                  &  & b_2b_3\cdots b_n0                  
\end{tikzcd}
\end{center}

\medskip

 \textbf{Case 2: } $s_1$ begins with 1. 

\medskip

 Let $s_1 = 1a_2a_3\cdots a_{n} \in \tilde N^+(n)$ with $\Xi(s_1) = 1b_2\cdots b_n.$ As in Case 1, it suffices to establish the result when $s_2 = L(s_1) = a_2a_3\cdots a_{n}1.$ In this case, $\Xi(s_2)_i = a_2 + \dots + a_{i+1} = \iota(b_{i+1})$ for $0\le i<n$ and (setting $a_1=1$) $\Xi(s_2)_n = a_1 + a_2 + \dots + a_n = 0$, meaning that $\Xi(s_2) = \iota(b_2) \iota(b_3) \cdots \iota(b_n) 0 = \iota(L(\Xi(s_1)))$.

\begin{center}
\begin{tikzcd}
1a_2a_3\cdots a_n \arrow[d, "\Xi"] \arrow[rr, "\mathrm{Rotate}"] &  & a_2a_3\cdots a_n 1 \arrow[d, "\Xi"] \\
1b_2b_3\cdots b_n \arrow[rr, "\textrm{Rotate and invert}"]                  &  & \iota(b_2)\iota(b_3)\cdots \iota(b_n)0   
\end{tikzcd}
\end{center}

We now establish (2) by proving that $\psi^{+}(s)$ is primitive and maps to the correct component of $\bar N_2(n)\ \dot\cup\ \bar N_1(n)$.
First, notice that $\Xi : \{ 0,1 \}^n \to \{ 0,1 \}^n$ is invertible, with
\[
\Xi ^ {-1}(s_1,\ldots,s_n) = (s_1, s_1 + s_2, s_2 + s_3, \dots, s_{n-1} + s_n),
\]
where all additions are performed modulo 2.

\medskip

 \textbf{Case 1: } $\langle s \rangle \in N^+(n).$

\medskip

Since $s$ is primitive, it has $n$ distinct cyclic left-shifts, each of which maps to a distinct binary string under $\Xi$. In addition, since $s$ contains an even number of $1$s, the last digit of $\Xi(s)$ is 0. If we invert each of these $n$ strings, 
we obtain $n$ distinct strings ending with $1$.
Therefore, the necklace inversion class of $s$ consists of $2n$ distinct strings, from which it follows that
$\psi^{+}(\langle s \rangle )\in \bar N_2(n).$ \\

\medskip

 \textbf{Case 2: } $\langle s \rangle \in N^-(n/2).$

\medskip

 Let $s=a_1a_2\cdots a_{n/2}a_1a_2\cdots a_{n/2}.$ We claim that $\Xi(s)$ is of the form 
 \[
 b_1b_2\cdots b_{n/2}\iota(b_1)\iota(b_2)\cdots \iota(b_{n/2}).
 \]

Indeed, $N^-(n/2)$ is defined so that there are an odd number of 1's in the first half of $s$, i.e., $b_{n/2}=a_1+a_2+\cdots+a_{n/2}=1$, and thus 
$$ \forall 1\le i\le n/2: b_{n/2+i}=b_{n/2}+(a_1+a_2+\cdots+a_i)=1 + b_i = \iota(b_i).$$
Note that $\iota(\Xi(s))$ is just $\Xi(s)$ shifted by $n/2$, from which it follows immediately that $\psi^{+}(\langle s \rangle )\in \bar N_1(n).$

\subsection{$\psi^+$ is bijective} \label{subsection:Xibijective}

In order to show that $\psi^+$ is bijective, we define a map $\theta^+ : \bar N(n) \to \tilde N^+(n)$ (which we will prove to be inverse to $\psi^+$) as follows: for $[\alpha] \in \bar N(n)$, take any representative $s \in [\alpha]$ ending in 0 and define $\theta^+([\alpha]) := \langle \Xi^{-1}(s) \rangle$.

\medskip

We will show that this map is well-defined, and moreover that
$\theta^+ : \bar N_2(n) \to N^+(n)$ and $\theta^+ : \bar N_1(n) \to N^-(n/2)$.

\medskip

 \textbf{Case 1: } $[\alpha] \in \bar N_2(n)$.

\medskip

As in the definition of $\theta^+$, take any representative $s \in [\alpha]$ ending in 0.
Since $[\alpha] \in \bar N_2(n)$, the necklace inversion class of $s$ consists of $2n$ distinct strings, $n$ of which end in $0$. We set $\nu_1 = [s]$ and label these strings 
 $\nu_{1}, \nu_{2},\ldots,\nu_{n}$ in such a way that if $\nu_i = [t]$ then 
 \[
 \nu_{i+1}=
 \left\{
 \begin{aligned}
 L(t) & \textrm{ if } L(t) \textrm{ ends in } 0 \\
 \iota(L(t)) & \textrm{ if } L(t) \textrm{ ends in } 1. \\ 
 \end{aligned}
 \right.
 \]
 For example, if $s=000100$ then 
\[
(\nu_{1}, \nu_{2},\ldots,\nu_{6}) = (000100, 001000, 010000, 100000, 111110, 000010).
\] 
 
 We generate this set by starting with any element of the class ending in 0 (here, 000100), then cyclically left-shifting $n$ times. Whenever the result ends in 1, then we invert the string so that it ends in 0 and then continue the process.

\medskip

 Note that since these strings all end in $0$, $\Xi^{-1}(\nu_i)$ has an even number of 1's for all $i$. 
 It remains to verify that $\Xi^{-1}$ takes 
 $\nu_1,\ldots,\nu_n$ to $n$ distinct rotations of the same string, thereby showing that $\langle s \rangle \in N^+(n)$ is primitive.
 
 It suffices, by symmetry, to show that $\Xi^{-1}(\nu_2) = L(\Xi^{-1}(\nu_1))$.

\medskip

 Let $\nu_1 = (s_1, s_2,\ldots,s_n)$, where $s_n=0$, so that $\Xi^{-1}(\nu_1)= (s_1, s_1 + s_2,\ldots , s_{n-1} + s_n)$.

\medskip

 \textbf{Case 1(a): }$s_{n-1} = 0$. We cyclically left-shift $\nu_1$ once to get $\nu_2 = (s_n, s_1, s_2,\ldots, s_{n-1}).$ A direct computation now shows that $\Xi^{-1}(\nu_2) = L(\Xi^{-1}(\nu_1))$.

\medskip

 \textbf{Case 1(b):} $s_{n-1} = 1$. In this case, we must rotate and also invert $\nu_1$ to get $\nu_2 = (\iota(s_n), \iota(s_1), \iota(s_2),\ldots,\iota(s_{n-1}))$. Again, a direct computation now gives $\Xi^{-1}(\nu_2) = L(\Xi^{-1}(\nu_1))$.

\medskip

Since $\psi^+$ is defined by applying $\Xi$ to an arbitrary representative string and $\theta^+$ is defined by applying $\Xi^{-1}$ to any representative string ending in $0$, it follows easily that $\psi^+ : N^+(n) \to \bar N_2(n)$ and $\theta^+ : \bar N_2(n) \to N^+(n)$ are inverse to one another.

\medskip

 \textbf{Case 2: } $[\alpha] \in \bar N_1(n)$.

\medskip
We again take any representative $s \in [\alpha]$ ending in 0.
In this case, Lemma~\ref{lem:primitive-and-reflexive} shows that $s$ is 2-alternating, and then the argument from Case 1 shows that
$\Xi^{-1}(s)$ consists of two copies of a primitive necklace $s'$ of length $n/2$.

It remains to check that $s'$ has an odd number of 1's. 
Since $s$ is 2-alternating, we have
$\iota(s_{n/2}) = s_{n}$. Moreover, we're given that $s_n = 0$, and thus $s_{n/2} = 1$. Since $s_{n/2}$ is the mod 2 sum of the first $n/2$ digits of $\Xi^{-1}(s)$, it follows that $s'$ has an odd number of 1's.
Thus $(\psi^{+})^{-1}: \bar N_1(n) \to N^-(n/2)$ is well-defined, and as in Case 1 it follows easily that
$\psi^+ : N^-(n/2) \to \bar N_1(n)$ and $\theta^+ : \bar N_1(n) \to N^-(n/2)$ are inverse to one another.

In particular, we conclude that $\psi^+: \tilde N^+(n) \to \bar N(n)$ is bijective, establishing Theorem~\ref{thm:N2N3}.

\section{The bijection between (P1) and (N2)}
\label{sec:N1N2}

Our goal in this section is to prove Theorems~\ref{thm:P1N2} and \ref{thm:N3P1}. 

\subsection{Definition of $\lambda$} \label{sec:lambda}

Our next goal is to define an inverse to the map $\phi : {\rm CUP}(n) \to \tilde{N}^+(n)$ defined in the introduction.
We first recall that $\phi$ is defined as follows. 

For $\sigma \in {\rm CUP}(n)$, let $m = m(\sigma)$ be the unique element of $[n]$ with $\sigma(m) = 1$. 
Define ${\rm Itin}(\sigma) \in \{ +, -, \star \}^n$ by setting ${\rm Itin}(\sigma)_n = \star$, and for $1 \leq i \leq n-1$ setting ${\rm Itin}(\sigma)_i = +$ if $\sigma^i(m) > m$ and ${\rm Itin}(\sigma)_i = -$ if $\sigma^i(m) < m$.
We define $A(\sigma) \in \{ 0,1 \}^n$ by
replacing all occurrences of $+$ in ${\rm Itin}(\sigma)$ with $0$, all occurrences of $-$ with $1$, and $\star$ by the unique element of $\{ 0,1 \}$ such that the mod-2 sum of the digits of $s$ is zero. 

We will show that the corresponding necklace ${\phi}(\sigma) := [A(\sigma)]$ belongs to $\tilde{N}^+(n)$, and 
the map $\phi : {\rm CUP}(n) \to \tilde{N}^+(n)$ is a bijection.

\medskip

We now define a map $\lambda : \bar{N}(n) \to {\rm CUP}(n)$ which we will show is inverse to $\phi$ (after identifying $\tilde{N}^+(n)$ and $\bar{N}(n)$ via the bijection $\psi^+: \tilde N^{+}(n)\rightarrow \bar N  (n)$ from Theorem~\ref{thm:N2N3}).

\medskip

Recall that $\bar N_1(n)$ consists of those necklace inversion classes $[s]$ such that $s$ and $\iota(s)$ generate the same necklace, and $\bar N_2(n) = \bar N(n)\setminus \bar N_1(n)$.

The definition of $\lambda$ is slightly more complicated on $\bar N_1(n)$, so we begin by defining $\lambda$ on $\bar N_2(n)$:

\begin{enumerate}
\item Given a necklace inversion class in $N_2(n)$, choose a binary string $t$ beginning with $1$ which represents this class.
\item Let $\mu_1 = t$, and for $2 \leq i \leq n$ let $\mu_i = F(\mu_{i-1})$, where $F$ is the 
twisted shift operator from Section~\ref{sec:symbolicdynamics}.
(These strings are all distinct by Lemma~\ref{lem:extendedformalism}.)
\item For $1 \leq i \leq n$ let $r_i$ denote the lexicographic ranking of $\mu_i$, i.e., one plus the number of $\mu_j$ which are lexicographically smaller than $\mu_i$.
\item Define $\lambda([t])$ to be the cyclic permutation $\sigma := (r_1 r_2 \cdots r_n)$.
\end{enumerate}

\begin{example}
Suppose we start with $[011101] \in N_2(6)$. Let $t = 100010$. Then
\[
(\mu_1,\ldots,\mu_6) = (100010,111010,110101,101011,101000,101110).
\]
The corresponding lexicographic ranking is $(1,6,5,3,2,4)$ (e.g., $100010$ is the smallest of the $\mu_i$ and $111010$ is the largest), and thus $\sigma = (165324)$.
\end{example}

For $[t] \in \bar N_1(n)$, the issue is that the strings $\mu_1,\ldots,\mu_n$ will no longer be distinct.

\begin{example}
Suppose we start with $[011100] \in N_1(6)$. Let $t = 100011$. Then
\[
(\mu_1,\ldots,\mu_6) = (100011,111000,110001,100011,111000,110001).
\]
\end{example}

In order to address this problem, when $[t] \in \bar N_1(n)$ we use the extended twisted shift operator $\tilde{F}$ from Section~\ref{sec:extendedsymbolicdynamics} instead of $F$.

\begin{enumerate}
\item Given a necklace inversion class in $\bar N_1(n)$, choose a binary string $t$ beginning with $1$ which represents this class.
\item Let $\mu_1 = t^-$, and for $2 \leq i \leq n$ let $\mu_i = \tilde{F}(\mu_{i-1})$. 
(These strings are all distinct by Lemma~\ref{lem:extendedformalism}.)
\item For $1 \leq i \leq n$, let $r_i$ denote the lexicographic ranking of $\mu_i$.
\item Define $\lambda([t])$ to be the cyclic permutation $\sigma := (r_1 r_2 \cdots r_n)$.
\end{enumerate}

\begin{example}
Suppose we start with $[011100] \in N_2(6)$. Let $t = 100011$. Then
\[
(\mu_1,\ldots,\mu_6) = (100011^-,111000^+,110001^+,100011^+,111000^-,110001^-).
\]
The corresponding lexicographic ranking is $(1,6,4,2,5,3)$ (e.g. $110001^+$ is the fourth-largest of the $\mu_i$), and thus
$\sigma = (164253)$.
\end{example}

\begin{theorem} \label{thm:P1N2N3}
\begin{enumerate}
    \item[]
    \item If $\sigma \in {\rm CUP}(n)$, the necklace ${\phi}(\sigma)$ belongs to $\tilde{N}^+(n)$.
    \item If $[t] \in \bar{N}(n)$, the permutation $\lambda([t])$ is independent of the choice of the representative string $t$.
    \item If $[t] \in \bar{N}(n)$, the permutation $\lambda([t])$ is unimodal.
    \item The composition $\phi \circ \lambda \circ \psi^+ : \tilde{N}^+(n) \to \tilde{N}^+(n)$ is the identity map.
    \item The map $\phi : {\rm CUP}(n) \to \tilde{N}^+(n)$ is injective.

\end{enumerate}
\end{theorem}

\begin{proof}
For (1), we first show that the output of $\phi$ cannot 
be $k$-primitive for $k>2$, with ``$k$-primitive'' meaning that a string has $k$ repeating sections.

Take a permutation $\sigma \in {\rm CUP}(n)$. Let $\sigma(m) = 1$. We define the \emph{$1$-region} of $\sigma$ as $\{1,2,...,m-1\}$, and the \emph{$0$-region} as $\{m+1,...,n\}$, with $m$ itself belonging to the $1$-region if $m$ is even and the $0$-region if $m$ is odd. Since $\sigma$ is unimodal, if $x,y \in \{1,2,...,n\}$ are both in the 1-region, $x<y \implies \sigma(x)>\sigma(y)$. Similarly, if $x, y$ are both in the 0-region, $x<y \implies \sigma(x)<\sigma(y)$. In  particular, if $x<y<z$ are in the same region, then either $\sigma(x) < \sigma(y) < \sigma(z)$ or $\sigma(z) < \sigma(y) < \sigma(x)$.

Assume for the sake of contradiction that ${\phi}(\sigma) = [s_1 s_2 \cdots s_n]$ is $k$-primitive with $k>2$. Then $s_j = s_{j+k}$ for all $1 \leq j \leq n$, where addition of indices is performed modulo $n$. It follows that for all $1 \leq i,j \leq n$ we have $\sigma^i(j) = \sigma^i(j+k)$. Since $k \geq 3$, we can choose $x < y < z$ in $\{ 1,\ldots, n \}$ which are congruent modulo $k$, and each of $x,y,z$
remains in the same region as we $\sigma$ repeatedly, with $\sigma^i(y)$  always between $\sigma^i(x)$ and $\sigma^i(z)$. Let $k \geq 0$ be the first iteration such that $\sigma^k(y) = m$. If $\sigma^k(x) <\sigma^k(y) = m < \sigma^k(z)$ then $\sigma^k(x)$ lies in the $1$-region while $\sigma^k(z)$ lies in the $0$-region, a contradiction. If $\sigma^k(z) <\sigma^k(y) = m < \sigma^k(x)$, we obtain an analogous contradiction. Thus $k \leq 2$ and ${\phi}(\sigma)$ is either primitive or 2-primitive. 

In the latter case, we must show each primitive half of ${\phi}(\sigma)$ has an odd number of $1$'s. Suppose not. Setting $x = 1$ and $y = \sigma^{n/2}(1)$, it follows from our 2-primitive assumption that $\sigma^{i}(x)$ and $\sigma^{i}(y)$ always remain in the same region. After $i = n/2$ iterations, we have encountered an even number of $1$s. Hence, since the the relative order flips an even number of times, the original order is preserved, i.e., $x < y \implies \sigma^{n/2}(x) < \sigma^{n/2}(y) $. Since $\sigma^{n/2}(x) = y$ and $\sigma^{n/2}(y) = \sigma^{n/2}(\sigma^{n/2}(1)) = 1$, we deduce that $y < 1$, a contradiction. 

\medskip

For (2), we consider the cases $[t] \in \bar{N}_2(n)$ and $[t] \in \bar{N}_1(n)$ separately.
We assume first that $[t] \in \bar{N}_2(n)$.
To show that $\lambda([t])$ is well-defined, suppose we select two distinct binary strings $t_1, t_2$, both beginning with 1's, that represent the class $[t]$. Then, $t_2 = F^j(t_1)$ for some $1 \leq j < n$. 
Write $\mu_1 = t_1$, and for $2 \le i \le n$ define $\mu_i = F^{i-1}(t_1)$. These $\mu_i$ have a lexicographic ranking which determines the permutation $\sigma.$ To see that $t_1$ and $t_2$ determine the same permutation, note that $t_2 = \mu_{j+1}.$ This means that applying $F$ repeatedly to $t_2$ gives the sequence $\mu_{j+1}, \mu_{j+2}, \ldots, \mu_{n}, \mu_1, \ldots, \mu_j$. The lexicographic ordering of these $\mu_i$ results in the same permutation $\sigma$ as before. 

The argument follows similarly for $[t] \in \bar{N}_1(n).$ Again, we take two distinct binary strings $t_1$ and $t_2$, both beginning with 1's to represent the class $[t].$ Note that $t_2 = \tilde{F}^j(t_1)$ for some $1 \leq j < n$. Write $\mu_1 = t_1^-$, and for $2 \le i \le n,$ define $\mu_i = \tilde{F}^{i-1}(t_1)$. These $\mu_i$ have a lexicographic ordering that determines the permuation $\sigma$. To see that $t_1$ and $t_2$ determine the same $\sigma$, note that $t_2 = \mu_{j+1},$ and applying $\tilde{F}$ repeatedly to $\mu_{j+1}$ gets us $\mu_{j+1}, \mu_{j+2}, \ldots, \mu_1, \ldots, \mu_j$. The lexicographic ordering of these $\mu_i$ results in the same permutation $\sigma$ as before. 

\medskip

For (3), we again consider the cases $[t] \in \bar{N}_2(n)$ and $[t] \in \bar{N}_1(n)$ separately.
We assume first that $[t] \in \bar{N}_2(n)$.
As in the above definition of $\lambda$ above, let $\mu_i=F^{i-1}(t)$. 

Define $C = (\alpha_1 \mapsto \alpha_2 \mapsto \cdots \mapsto \alpha_n) \in T_n'$,
where $\beta(\alpha_i) = \mu_i$ 
(e.g., if $\mu_i=10110,$ then $\alpha_i=.\overline{10110}$).
Then
$\tau(\{ C, I(C) \}) = ([t])$, where $\tau$ is the map defined in Corollary~\ref{cor:D1toN3}.

By the discussion in Section~\ref{sec:symbolicdynamics},
applying $F$ to $\mu_i$ corresponds to applying $T'$ to $\alpha_i$.
Since $T'$ is unimodal and $\beta$ is order-preserving, it follows that $\sigma$ is unimodal as well.

Our reasoning from Case 1 works when $[t] \in \bar{N}_1(n)$ as well,
provided that we replace $F$ with $\tilde{F}$, $\mu_i$ with the corresponding extended binary string (i.e., $\mu_i = \tilde{F}^{i-1}(t^-)$), and 
$\alpha_i$, $\beta$, and $T'$ with their extended counterparts.
The point is that $\tilde{\beta}$ is also order-preserving and, by Lemma~\ref{lem:extendedunimodal}, $\widetilde{T'}$ is also unimodal.

\medskip

For (4), let $x=\left< s \right>$ be a necklace in $\tilde{N}^+(n)$. Let $y = \psi^+(x) = [t]$ be the corresponding element of $\bar{N}(n)$, where $t=\Xi(s)$. Let $\sigma = \lambda(y) \in {\rm CUP}(n)$, and let $z = {\phi}(\sigma) = \left< u \right>$ be in $\tilde{N}^+(n)$, where $u = A(\sigma)$. We wish to show that $x=z$.

We again consider the cases $y \in \bar{N}_2(n)$ and $y\in \bar{N}_1(n)$ separately.
We assume first that $y \in \bar{N}_2(n)$.
As in the definition of $\lambda$, let $\mu_i = F^{i-1}(t)$. Since $\psi^+$ is bijective and can be computed using any string representing the necklace $x$, we may assume without loss of generality that $\mu_2$ is the lexicographically smallest string among $\mu_1,\ldots,\mu_n$, i.e., $r_2 = 1$. 

Define $C = (\alpha_1 \mapsto \alpha_2 \mapsto \cdots \mapsto \alpha_n) \in T_n'$ as before,
where $\beta(\alpha_i) = \mu_i$.
Since $\mu_2$ is lexicographically smallest among the $\mu_i$, it follows that $\alpha_2$ is closest in value to $1/2$ among all $\alpha_i$. Therefore the value preceding it, $\alpha_1$, is the closest in value to $(T')^{-1}(1/2)=3/4$. The indices $i>1$ for which $u_i=1$ are precisely those for which $r_i<r_1$, which means that $\mu_i$ comes lexicographically before $\mu_1$. This happens iff $\alpha_i<\alpha_1$, which in turn happens iff $\alpha_i<3/4$ iff $\mu_i$ starts with $10$. 

Since $\mu_i = F^i(\Xi(s))$, it follows from the definitions that for $i > 1$, $\mu_i$ starts with $10$ iff $s_{i+1}=1$, where the index is computed modulo $n$. Therefore, for $i \neq 1$, we have $u_i=1$ precisely when $s_{i+1}=1$. But since $x$ and $z$ both have an even number of 1's, we must also have $u_1 = s_2$ as well, and thus $u = L(s)$. In particular, $u$ and $s$ represent the same necklace, so $x=z$ as desired.

If $y \in \bar{N}_1(n)$, our reasoning from Case 1 still works,
provided that we replace $F$ with $\tilde{F}$, $\mu_i$ with the corresponding extended binary string,
and $\alpha_i$ with its extended version.
The one additional thing we need to verify is that, in this extended setting, $\alpha_2$ being the closest to $1/2$ still implies that $\alpha_1$ is the closest to $3/4$.
This is a straightforward consequence of Lemma~\ref{lem:extendedunimodal}.

\medskip

For (5), suppose we have two permutations $\sigma_1,\sigma_2 \in {\rm{CUP}}(n)$ such that $\phi(\sigma_1)=\phi(\sigma_2) \in \tilde{N}^+(n)$. We wish to show that $\sigma_1=\sigma_2.$
We again assume that $\phi(\sigma_1)=\phi(\sigma_2) \in \tilde{N}_2(n)$, leaving the $\tilde{N}_1(n)$ case as a straightforward extension.

By hypothesis, $A(\sigma_1)$ and $A(\sigma_2)$ are rotations of the same necklace, so $\Xi(A(\sigma_1))$ and $\Xi(A(\sigma_2))$ yield the same necklace inversion class in $\bar N(n)$. Using Lemma \ref{lem:P1N2N3} below, we see that when inverted (as necessary) in order to begin with $1$, these two strings must be identical, since they share an orbit under $F$ (because $F$ cycles through every string in the necklace inversion class which starts with $1$), and both must be the lexicographically smallest element in that orbit. 

So $\Xi(A(\sigma_1))$ and $\Xi(A(\sigma_2))$ are identical up to inversion, meaning that $A(\sigma_1)$ and $A(\sigma_2)$ are identical except for possibly the first digit. But they are both in $\tilde{N}^+(n),$ so they each have an even number of $1$'s, making it impossible for them to differ in only one digit. Thus $A(\sigma_1)=A(\sigma_2)$, which implies that ${\rm Itin}(\sigma_1)={\rm Itin}(\sigma_2)$. In other words, $\sigma_1$ and $\sigma_2$ share the same itinerary.

It remains to show that a cyclic unimodal permutation $\sigma$ is uniquely determined by its itinerary ${\rm Itin}(\sigma)$. We do this by showing how to determine whether $\sigma^i(m)<\sigma^j(m)$ or $\sigma^i(m)>\sigma^j(m)$ for all $i<j$ in $[n]$.

We use induction on $j'=n-j$. The base case $j'=0$ occurs when $j=n$, so $\sigma^j(m)=m$, and the itinerary directly tells us whether $\sigma^i(m)$ is greater or less than $m$. For the inductive step, note that if ${\rm Itin}(\sigma)_i=+$ and ${\rm Itin}(\sigma)_j=-$, then $\sigma^i(m) > m > \sigma^j(m)$, and if ${\rm Itin}(\sigma)_i=-$ and ${\rm Itin}(\sigma)_j=+$, then $\sigma^i(m) < m < \sigma^j(m)$.

Otherwise, if ${\rm Itin}(\sigma)_i=+$ and ${\rm Itin}(\sigma)_j=+$, then since $\sigma$ is increasing on $[m,n]$, we know $\sigma^i(m) < \sigma^j(m)$ if and only if $\sigma^{i+1}(m) < \sigma^{j+1}(m)$, and we can use the inductive hypothesis. If ${\rm Itin}(\sigma)_i=-$ and ${\rm Itin}(\sigma)_j=-$, then since $\sigma$ is decreasing on $[1,m]$, we know $\sigma^i(m) < \sigma^j(m)$ if and only if $\sigma^{i+1}(m) > \sigma^{j+1}(m)$, and we can again apply the inductive hypothesis.

\end{proof}

The following Lemma was used in the proof of part (5) of Theorem \ref{thm:P1N2N3}:

\begin{lemma} \label{lem:P1N2N3}
    Let $\sigma \in {\rm{CUP}}(n)$ and let $t=\Xi(A(\sigma))$.
    Then either $t$ or $\iota(t)$, whichever begins with the digit $1$, is the lexicographically smallest element in its orbit under $F$ (if $[t] \in \bar N_2(n)$) or $\tilde{F}$ (if $[t] \in \bar N_1(n)$).
\end{lemma}

\begin{proof}
Let $\sigma \in {\rm{CUP}}(n)$, let $s = A(\sigma)$, and let $t = \Xi(s)$. 
We assume for simplicity that $[t] \in \bar N_2(n)$; the case $[t] \in \bar N_1(n)$ is entirely analogous.

Let $\mu_1 = t$ or $\iota(t)$, whichever starts with the digit $1$, and let $\mu_i = F(\mu_{i-1})$ for $2 \le i \le n$. We wish to show that $\mu_1$ is lexicographically minimal among all the $\mu_i$. (It is possible, if $[t] \in \bar N_1(n)$, to have repetition among the $\mu_i$, but that will not present any issue for us.)

By the definition of $F$, all the $\mu_i$ start with the digit $1$. This means that we want $\mu_1$ to have a $1$, then the most $0$s, then the fewest $1$s, then the most $0$s, etc. out of all the $\mu_i$
(with earlier conditions taking precedence).
Equivalently, among all cyclic rotations of $s$, 
we want the string $s$ to have any first digit, then a $1$, then the most $0$s, then a $1$, then the fewest $0$s, then a $1$, then the most $0$s, etc. (again, with earlier conditions taking precedence).

Let $m$ be the element of $[n]$ such that $\sigma(m)=1$, and define the $0$-region and $1$-region of $\sigma$ as in the proof of part (1) of Theorem \ref{thm:P1N2N3}. 
With this terminology,
we need to show that, as we vary over all starting points in $[n]$, $m$ is the value that maps to the $1$-region, then maximizes time in the $0$-region, then maps to the $1$-region, then minimizes time in the $0$-region, etc.

Recall also from part (1) of Theorem \ref{thm:P1N2N3} that since $\sigma$ is unimodal, if $x$ and $y$ are both in the $1$-region, $x<y$ implies $\sigma(x) > \sigma(y)$, and if $x$ and $y$ are both in the $0$-region, $x<y$ implies $\sigma(x) < \sigma(y)$.

Suppose $x \neq m$ and $1 = \sigma(m) < \sigma(x)$, with $\sigma(m)$ and $\sigma(x)$ both in the $1$-region. Then $\sigma(\sigma(m))>\sigma(\sigma(x))$. If both are in the $0$-region, then applying $\sigma$ again will preserve their order but decrease the values of both, so that eventually (say at the $i$th iteration of $\sigma$) at least one reaches the $1$-region. After this process, if $\sigma^i(x)$ is in the $1$-region while $\sigma^i(m)$ is not, then we are done, since the itinerary of $m$ spent longer in the $0$-region than the itinerary of $x$. 

If they both enter the $1$-region on the same iteration, then they are ``tied”, and we move on to the next part of the condition: minimizing time in the $0$-region in the next iterations after $\sigma^i$. We still have $\sigma^i(m) > \sigma^i(x)$, but both are in the $1$-region, so $\sigma^{i+1}(m)<\sigma^{i+1}(x)$. If both are in the $0$-region, then repeatedly applying $\sigma$ will preserve their order while reducing the values until at least one of them enters the $1$-region. Suppose this occurs on the $j$th iteration of $\sigma$. If only $\sigma^j(m)$ is in the $1$-region, then we are done, since it spent less time in the $0$-region than $\sigma^j(x)$. If both are in the $1$-region, then they are tied again. But $\sigma^j(m)<\sigma^j(x)$ implies $\sigma^{j+1}(m)>\sigma^{j+1}(x)$, 
and we can continue with the same reasoning we used at the start.

The only case left is if we only have ties all the way down the line, in which case $\mu_1 = \mu_k$ for some $1 < k \leq n$, but then $\mu_1,\ldots,\mu_n$ are not distinct, violating Lemma~\ref{lem:extendedformalism}.
\end{proof}

\begin{proof}[Proofs of Theorem \ref{thm:P1N2} and Theorem \ref{thm:N3P1}]
From parts $(1)$, $(2)$, and $(3)$ of Theorem \ref{thm:P1N2N3}, we see that $\phi$ and $\lambda$ are both well-defined. From part $(4)$, we know that $\phi : {\rm CUP}(n) \to \tilde{N}^+(n)$ is surjective: given any necklace $x$ in $\tilde{N}^+(n)$, applying $\lambda \circ \psi^+$ yields a permutation in ${\rm CUP}(n)$ which maps to $x$ under $\phi$. Combining this with part $(5)$ shows that $\phi$ is bijective, completing the proof of Theorem \ref{thm:P1N2}.

Having established that $\phi$ and $\psi^+$ are both bijections, part $(4)$ of Theorem \ref{thm:P1N2N3} shows that $\lambda$ is also a bijection, completing the proof of Theorem \ref{thm:N3P1}.
\end{proof}

\begin{remark}
We have the following concrete description of the partition 
\[
{\rm CUP}(n) = {\rm CUP}_1(n) \ \dot\cup \ {\rm CUP}_2(n) 
\]
induced by our bijection between ${\rm CUP}(n)$ and $\tilde N^+(n)$ (Theorem \ref{thm:P1N2}) and the natural decomposition $\tilde N^+(n) = \tilde N_1(n) \ \dot\cup \ \tilde N_2(n) = N^-(n/2)\ \dot\cup \ N^+(n)$.

Namely, a permutation $\sigma \in {\rm CUP}(n)$ belongs to ${\rm CUP}_1(n)$ iff $n$ is even and 
$\sigma$ has cycle decomposition
\[
(1=x_1 \mapsto x_2 \mapsto \cdots \mapsto x_{n/2} \mapsto y_1 \mapsto y_2 \mapsto \cdots \mapsto y_{n/2}),
\]
where 
either $x_i$ is odd and $y_i=x_i + 1$ or $x_i$ is even and $y_i = x_i - 1$ for all $1 \leq i \leq n$.
We omit the proof, whose main ingredient is the description of a permutation $\sigma \in \tilde N_1(n)$ in terms of the $\tilde{F}$-orbit of the corresponding necklace inversion class in $\bar N_1(n)$.

For example, the permutation $\sigma = (1 6 4 2 5 3)$ belongs to $\tilde N_1(6)$, and this satisfies the required conditions with $x_1 = 1, x_2 = 6, x_3 = 4$.
\end{remark}

\section{Examples} \label{sec:examples}

We illustrate various bijections discussed in this paper in the cases $n=4$, $n=5$, and $n=6$.

For the equivalence with (M2), which requires choosing a normal basis of $\FF_{2^n} / \FF_2$, we make the following choices:

\begin{enumerate}
    \item When $n=4$, we let $\alpha$ be a root of $x^4 + x + 1 \in \FF_2[x]$, we let $\beta = \alpha^3$, and we take $\beta^8,\beta^4,\beta^2,\beta$
    as our normal basis.
    \item When $n=5$, we let $\alpha$ be a root of $x^5 + x^2 + 1 \in \FF_2[x]$, we let $\beta = \alpha^3$, and we take 
    $\beta^{16},\beta^8,\beta^4,\beta^2,\beta$ as our normal basis.
    \item When $n=6$, we let $\alpha$ be a root of $x^6 + x + 1 \in \FF_2[x]$, we let $\beta = \alpha^5$, and we take 
    $\beta^{32},\beta^{16},\beta^8,\beta^4,\beta^2,\beta$ as our normal basis.
\end{enumerate}

The other bijections are canonical, and do not require such a choice.

\begin{table}[ht] \caption{Bijections for $n=4$} 
 \centering    
 \begin{tabular}{c | c | c | c | c | c | c}  
(M1) & (M2) & (D1) & (P1) & (N1) & (N2) & (N3)   \\
 \hline\hline   
 $-1.9408\ldots$ & $x^4 + x + 1$ & $7/15$ & (1432) & 1000 & 1100 & 0111  \\ 
\hline
$-1.3107\ldots$  & $x^2 + x + 1$ & $6/15$ & (1423) & 1011 & 0101 & 0110  \\
 \hline     \end{tabular} 
  \end{table}
  
\begin{table}[ht] \caption{Bijections for $n=5$} 
 \centering    
 \begin{tabular}{c | c | c | c | c | c | c}  
(M1) & (M2) & (D1) & (P1) & (N1) & (N2) & (N3)   \\
 \hline\hline   
$-1.9854\ldots$ & $x^5 + x^2 + 1$ & $15/31$ & (15432) & 10000 & 11000 & 01111  \\ 
 \hline
$-1.8607\ldots$ & $x^5 + x^3 + x^2 + x + 1$ & $14/31$ & (15423) & 10011 & 01001 & 01110 \\
 \hline
$-1.6254\ldots$ & $x^5 + x^3 + 1$ & $13/31$ & (15324) & 10110 & 11011 & 01101 \\
 \hline     \end{tabular} 
  \end{table} 

\begin{table}[ht] \caption{Bijections for $n=6$} 
 \centering    
 \begin{tabular}{c | c | c | c | c | c | c}  
(M1) & (M2) & (D1) & (P1) & (N1) & (N2) & (N3)   \\
 \hline\hline   
$-1.9963\ldots$ & $x^6 + x + 1$ & $31/63$ & (165432) & 100000 & 110000 & 011111  \\ 
 \hline
$-1.9667\ldots$ & $x^6 + x^3 + 1$ & $30/63$ & (165423) & 100011 & 010001 & 011110 \\
 \hline
$-1.9072\ldots$ & $x^6 + x^4 + x^2 + x + 1$ & $29/63$ & (165324) & 100110 & 110011 & 011101 \\
 \hline     
 $-1.4760\ldots$ & $x^6 + x^4 + x^3 + x + 1$ & $26/63$ & (163425) & 101111 & 010111 & 011010 \\
 \hline    
 $-1.7728\ldots$ & $x^3 + x^2 + 1$ & $28/63$ & (164253) & 100101 & 010010 & 011100  \\
 \hline    
 \end{tabular} 
  \end{table} 

In column (N3), corresponding to the set $\bar N(n)$, we choose a string $t$ representing a given necklace inversion class $\alpha$ so that (a) $t$ begins with 0, and (b) $\iota(t)$ is lexicographically minimal among all strings in $\alpha$ beginning with $1$.

We obtain the corresponding string in column (N2) by applying $\Xi^{-1}$ to $\iota(t)$. The string in column (N1) is obtained by applying the Weiss-Rogers bijection, $\Psi : {\rm CUP}(n) \to N^-(n)$, from Section 2.1 to column (P1). 

We obtain the corresponding entry in column (D1) by computing the fraction $\theta$ whose periodic binary expansion is given by $t$ (e.g., $31/63 = .\overline{011111}$).

We obtain the entry $c$ in column (M1) from the corresponding entry $\theta$ in column (D1) using the tables from \cite[Appendix A]{BFKP}, together with the observation that 
$\theta$ is the kneading angle corresponding to $c$ (see the discussion in Section~\ref{sec:MT} below).

\section{Complements}

\subsection{Relationship to itineraries and Milnor and Thurston's kneading invariant} \label{sec:MT}

Real numbers $c$ for which the orbit of the critical point $0$ under $f_c(z) = z^2 + c$ has primitive period $n$ are known in the literature as {\em real hyperbolic centers}.
These are the same thing as real roots of the $n^{\rm th}$ Gleason polynomial $G_n$.

The {\em itinerary} of a real hyperbolic center $c$ of primitive period $n$ is the sequence $\kappa(c) = (\kappa_0,\kappa_1,\kappa_2,\ldots,\kappa_{n-1})$, with  $\kappa_i \in \{ +, -, \star \}$ for all $i$, defined by 
\[
\kappa_i = 
\begin{cases}
+ & \text{if $f_c^i(0)>0$} \\
\star & \text{if $f_c^i(0)=0$} \\
- & \text{if $f_c^i(0)<0$}. \\
\end{cases}
\]

The {\em kneading angle sequence} of $c$ is the binary string $t_1 t_2 \cdots t_n$ defined by
\[
t_{i+1} = 
\begin{cases}
0 & \text{if $\kappa_i = 0$} \\
t_i & \text{if $\kappa_i = +$} \\
1-t_i & \text{if $\kappa_i = -$}. \\
\end{cases}
\]

Finally, the \emph{kneading angle} of $c$ is the real number $\theta(c) \in [0,1/2]$ defined by
\[
\theta(c) = \frac{\sum_{j=1}^n t_j 2^{n-j}}{2^n - 1} = 0.\overline{t_1 t_2 \cdots t_n}.
\]

\begin{example}
For $n=3$, the polynomial $G_3(c)=c^3 + 2c^2 + c + 1$ has a unique real root $c = -1.754878\ldots$ with itinerary $(\star,-,+)$ corresponding to the orbit of the critical point $0$. The corresponding kneading angle sequence is $(0,1,1)$, which yields the kneading angle $\theta(c) = 0.\overline{011} = 3/7$.
\end{example}

The work of Milnor and Thurston (cf.~\cite{Milnor-Thurston} and \cite[Theorem 4.3]{BFKP}) implies:

\begin{theorem} \label{thm:MilnorThurston}
\begin{enumerate}
    \item[]
    \item If $c_1, c_2$ are real hyperbolic centers, then $c_1 < c_2$ if and only if $\theta(c_1) > \theta(c_2)$. 
    \item A real number $\theta \in [0,1/2)$ of period $n$ for $D$ is the kneading angle of some real hyperbolic center of period $n$ iff $\theta$ is the closest angle to $1/2$ within its orbit under iteration of $D$.
    \item If $c$ is a real hyperbolic center of primitive period $n$, then $\theta(c)$ has primitive period $n$ for the doubling map $D : \RR/\ZZ \to \RR/\ZZ$. 
\end{enumerate}
\end{theorem}

The correspondence with the formalism in this paper is as follows.
Given a real hyperbolic center $c$ of primitive period $n$, let $\kappa$ be its itinerary, let $t=t_1 t_2 \cdots t_n$ be its kneading angle sequence, and let $[t] \in \bar N(n)$ be the necklace inversion class of $t_1 t_2 \cdots t_n$.

By combining the work of Milnor--Thurston with the results proved in this paper, one can show:

\begin{enumerate}
    \item If we list the elements of the critical orbit $\{ f_c(0), f_c^2(0), \ldots, f_c^{n-1}(0), f_c^n(0)=0 \}$ of $f_c$ in increasing order, with $f_c^i(0)$ having rank $r_i$, the resulting permutation $\sigma = (r_1 r_2 \cdots r_n)$ is equal to 
    $\lambda([t])$, where $\lambda$ is the map defined in Section~\ref{sec:lambda}. Moreover, $r_1 = 1$, i.e., $f_c(0)$ is the smallest element in this orbit.
    \item The itinerary of $c$ is equal to the itinerary ${\rm Itin}(\sigma)$ (as defined in Section~\ref{sec:NecklaceBijections}), cyclically right-shifted so that it begins with $\star$.
    \item The binary string $s = A(\sigma)$, as defined in Section~\ref{sec:NecklaceBijections}, satisfies $\Xi(s)=t$. In particular, the corresponding necklace $\nu = \langle A(\sigma) \rangle \in \tilde N^+(n)$ satisfies $\psi^+(\nu)=[t]$.
    \item $\iota(t)$ is lexicographically minimal among all strings in $[t]$ starting with 1. (This corresponds to Theorem~\ref{thm:MilnorThurston}(3).)
\end{enumerate}

\begin{example}
Let $n=4$. We have already seen that the necklaces in $\tilde{N}^+(4)$ are represented by the primitive string $0011$ and the non-primitive string $0101$.

\begin{enumerate}
\item For $s=0011$, we have $\Xi(s)=0010$ and the lexicographically minimal rotation which starts with $1$ is $1000$. Its complement is $t=\iota(1000)=0111$. The corresponding kneading invariant is $\theta = .\overline{0111}=7/15.$ The $F$-orbit of $1000$ is
\[
1000 \mapsto 1110 \mapsto 1101 \mapsto 1011,
\]
and ordering these strings lexicographically gives the cyclic unimiodal permutation $(1 4 3 2)$. The corresponding itinerary is $(-,+,+,\star)$.

The counterpart to this story in terms of iteration of $z^2 + c$ is as follows.
According to \cite[Appendix A]{BFKP}, $\theta=7/15$ is the kneading angle of the real hyperbolic 
center $c = -1.940800\ldots$
The critical orbit of $z^2 + c$ is 
\[
c = -1.940800\ldots \mapsto 1.825904\ldots \mapsto 1.393125\ldots \mapsto 0,
\]
whose associated itinerary is $(\star,-,+,+)$.
Since 
\[
-1.940800\ldots < 0 < 1.393125\ldots < 1.825904\ldots,
\]
the permutation associated to the critical orbit is $(1 4 3 2)$.

\item For $s=0101$, we have $\Xi(s)=0110$, and the lexicographically minimal rotation which starts with $1$ is $1001$. Its complement is $t=\iota(1001)=0110$. The corresponding kneading invariant is $\theta = .\overline{0110}=6/15.$ The $\tilde{F}$-orbit of $1001^-$ is
\[
1001^- \mapsto 1100^+ \mapsto 1001^+ \mapsto 1100^-,
\]
and ordering these strings lexicographically gives the cyclic unimiodal permutation $(1 4 2 3)$. The corresponding itinerary is $(-,+,-,\star)$.

On the dynamical side, according to \cite[Appendix A]{BFKP}), $6/15$ is the kneading angle of the real hyperbolic center $c = -1.310703\ldots$
The critical orbit of $z^2 + c$ is 
\[
c = -1.310703\ldots \mapsto 0.407239\ldots \mapsto -1.144859\ldots \mapsto 0,
\]
whose associated itinerary is $(-,+,-,\star)$.
Since 
\[
-1.310703\ldots <  -1.144859\ldots < 0  <  0.407239\ldots 
\]
the permutation associated to the critical orbit is $(1 4 2 3)$.
\end{enumerate}
\end{example}

\begin{remark}
There are two kinds of real centers, those that correspond to {\em primitive} hyperbolic components of the Mandelbrot set and those which correspond to {\em satellite} components (cf.~\cite{BFKP}). It is easy to see that, under the bijection ${\rm (M1)} \leftrightarrow {\rm (M2)}$ established in this paper, the primitive centers correspond to the irreducible factors of $G_n$ mod 2 of degree $n$,
and the satellite centers correspond to the irreducible factors of degree $n/2$.
\end{remark}

\subsection{Additional questions}

Here are some interesting questions left unanswered by the present paper:

\begin{enumerate}
    \item Let $S_1(n)$ denote the collection of subsets of $\{1,2,\cdots,n-1\}$ such that the sum of the elements of the subset is $1\bmod n$. 
    According to the OEIS entry for sequence A000048, $S_1(n)$ has cardinality $\gamma_n$ for all $n\geq 2$. Is there an explicit family of bijections between $S_1(n)$ and any of the sets described in Section \ref{sec:intro}?
    More generally, define $S_1(n,k)$ to be the collection of $k$-element sets belonging to $S_1(n)$, and let
    \[
    \text{CUP}_k(n)=\{\sigma\in \text{CUP}(n)\mid\sigma(k)=1\}.
    \]
    We conjecture that $|S_1(n,k-1)| = | \text{CUP}_k(n)|$, and ask if there is an explicit family of bijections between $S_1(n,k-1)$ and $\text{CUP}_k(n)$. We have verified the conjecture concerning cardinalities for $n\leq 10$.
    \item Sequence A000016 in OEIS is given by a similar formula to $\gamma_n,$ with the Möbius function $\mu$ replaced by the Euler totient function $\phi$. It enumerates, among other things, the set $T^-(n)$ of (not-necessarily-primitive) binary necklaces with an odd number of 1's. It is also the cardinality of the collection $S_0(n)$ of subsets of $\{1,2,\ldots,n-1\}$ the sum of whose elements is congruent to 0 modulo $n.$ Can we construct an explicit bijection between $S_0(n)$ and $T^-(n)$? Also, can we explain the striking relationship between the enumerative formulas for $S_0(n)$ and $S_1(n)$?
\end{enumerate}

\appendix 

\section{Cardinality of $\bar N(n)$} \label{sec:N3card}

A large portion of this paper consists of establishing bijections between various sets, all of which have cardinality $\gamma_n$.
For completeness, we give a proof in this appendix that the set $\bar N(n)$ of primitive necklace inversion classes of length $n$ indeed has cardinality $\gamma_n$ (and therefore so do all the other sets in bijection with it). The argument is sketched in \cite[Proof of Theorem 4.1]{Miller}, but we provide more details.

Let $p_n$ (resp. $c_n$) denote the number of primitive binary strings (resp. primitive binary necklaces) of length $n$. By M\"obius inversion, we have
$$p_n = \sum_{d \mid  n} \mu(n/d)2^d$$
and therefore
$$c_n = \frac{1}{n}\sum_{d \mid  n} \mu(n/d)2^d.$$

\medskip

Our derivation will use the following simple Lemma:

\begin{lemma} \label{lem:appendixlemma}
Suppose $s=s' \iota(s')$ is 2-alternating (cf.~Section~\ref{subsection:Xibijective}). Then $s$ is non-primitive if and only if $s'$ is $m$-alternating, where $m$ is an odd divisor of $n/2$, i.e., 
$s'=s_1 \iota(s_1) s_1 \iota(s_1) \cdots s_1$ with $s_1$ of length $\frac{n}{2m}$. 
\end{lemma} 

\begin{proof}
The ``if'' direction is trivial, so it suffices to prove the ``only if'' direction.
As $s$ is non-primitive, there is a unique primitive string $t$ of length $d$ which repeats $m := n/d \geq 2$ times to form $s=(s_1,\ldots,s_n)$.

If $d \mid \frac{n}{2}$, then
 \[
 (s_{1}, s_2,\ldots,s_{n/2})  = (s_{1 +n/2}, s_{2+ n/2},\ldots, s_{n - 1}),
 \]
contradicting the definition of 2-alternating. We conclude that $d \nmid \frac{n}{2}$, and therefore
$m = \frac{n}{d}$ is odd.

\medskip

Since $s=s' \iota(s')$, it follows that $t = t' \iota(t')$ for some string $t'$ as desired. 
\end{proof}
 
\begin{theorem}
There are 
\[
\gamma_n = \frac{1}{2n} \sum_{m\mid n, \, 2 \nmid m} \mu(m) 2^{\frac{n}{m}}
\]
primitive necklace inversion classes of length $n$.
\end{theorem}

\begin{proof}
We consider separately the cases where $n$ is odd or even.

\medskip

 \textbf{Case 1:} $n$ is odd.

\medskip
 
 Inverting a binary necklace switches the parity of the number of 1's, so the inversion-equivalence class $[s]$ of a primitive necklace $\langle s \rangle$ of odd length $n$ always contains $2$ distinct elements.
 Thus 
$$ |\bar N(n)| = \frac{c_n}{2} = \frac{1}{2n} \sum_{d \mid  n} \mu(n/d)2^d = \gamma_n.$$ 

\medskip

 \textbf{Case 2:} $n$ is even.

 \medskip
 
 Let $\epsilon_n$ be the number of primitive necklaces of length $n$ fixed under inversion, and let $\delta_n$ be the number of $2$-element inversion classes. The total number of inversion classes is then $\epsilon_n + \delta_n$. We want to show that this equals $\gamma_n$. 

 We have
 $$n\epsilon_n+2n\delta_n = p_n,$$ 
and thus
 $$(\epsilon_n+\delta_n) =  \frac{p_n+\xi_n}{2n},$$
where
$\xi_n = n \epsilon_n$ is the number of strings of length $n$ which are both primitive and reflexive.

By Lemma~\ref{lem:primitive-and-reflexive}, such a string must be 2-alternating. 
There are $2^{d/2}$ different $2$-alternating strings of length $d$, so by Lemma~\ref{lem:appendixlemma},
\[
2^{n/2} = \sum_{m \mid \frac{n}{2}, \, 2\nmid m} \xi_m.
\]
By M\"obius inversion, we obtain
$$\xi_n = \sum_{m \mid \frac{n}{2}, \, 2\nmid m} \mu(m) 2^{\frac{n}{2m}}.$$ 

It follows that
\begin{align*}
    |\bar N(n)| &= \epsilon_n+\delta_n \\
    &= \frac{p_n+\xi_n}{2n} \\
    &= \frac{1}{2n} \sum_{m\mid n} \mu(m) 2^{\frac{n}{m}}+\frac{1}{2n} \sum_{m\mid \frac{n}{2}, \, 2 \nmid m} \mu(m) 2^{\frac{n}{2m}}\\
    &= \frac{1}{2n} \sum_{m\mid n} \mu(m) 2^{\frac{n}{m}}+\frac{1}{2n} \sum_{\substack{m\mid n \\ m\equiv2 \, ({\rm mod \,} 4)}} \mu(m/2) 2^{\frac{n}{m}}\\
    &= \frac{1}{2n} \sum_{m\mid n} \mu(m) 2^{\frac{n}{m}}-\frac{1}{2n} \sum_{\substack{m\mid n \\ m\equiv2 \, ({\rm mod \, } 4)}} \mu(m) 2^{\frac{n}{m}}.\\
    &= \frac{1}{2n} \sum_{m\mid n, \, 2 \nmid m} \mu(m) 2^{\frac{n}{m}}
\end{align*}

as desired.
\end{proof}

\begin{remark}
One can reformulate the above argument in terms of counting orbits of group actions, as is done in \cite{Miller}. The idea is that one can view necklace inversion classes as $G\times H$-orbits of the space of functions $f : \{ 1,\ldots, n \} \to \{ 0,1 \}$, where $G=\ZZ / n\ZZ$ acts on $\{ 1,\ldots, n \}$  by translation mod $n$ and $H=\ZZ / 2\ZZ$ acts on $\{ 0,1 \}$ by translation mod 2. Using Burnside's Lemma, counting orbits can be reduced to the problem of counting fixed points. Polya's method uses this framework to count $G$-orbits of functions $f : X \to Y$ induced by a $G$-action on $X$, and the more general setup where one has group actions on both the source and target is known as DeBruijn's method. For the application at hand, it's easy enough to compute things directly and so we've opted not to 
invoke DeBruijn's method.
\end{remark}


\bibliographystyle{plain}
\bibliography{PROMYS}

\begin{thebibliography}{10}

\bibitem{Archer-Elizalde}
Kassie Archer and Sergi Elizalde.
\newblock Cyclic permutations realized by signed shifts.
\newblock {\em J. Comb.}, 5(1):1--30, 2014.

\bibitem{BFKP}
Xavier Buff, William Floyd, Sarah Koch, and Walter Parry.
\newblock Factoring {G}leason polynomials modulo 2.
\newblock {\em J. Th\'{e}or. Nombres Bordeaux}, 34(3):787--812, 2022.

\bibitem{Collet-Eckmann}
Pierre Collet and Jean-Pierre Eckmann.
\newblock {\em Iterated maps on the interval as dynamical systems}, volume~1 of
  {\em Progress in Physics}.
\newblock Birkh\"auser, Boston, MA, 1980.

\bibitem{Devaney}
Robert~L. Devaney.
\newblock {\em An introduction to chaotic dynamical systems}.
\newblock Addison-Wesley Studies in Nonlinearity. Addison-Wesley Publishing
  Company, second edition, 1989.

\bibitem{Diaconis-Graham}
Persi Diaconis and Ron Graham.
\newblock {\em Magical mathematics}.
\newblock Princeton University Press, Princeton, NJ, 2012.
\newblock The mathematical ideas that animate great magic tricks, With a
  foreword by Martin Gardner.

\bibitem{Douady-Hubbard}
A.~Douady and J.~H. Hubbard.
\newblock {\em \'Etude dynamique des polyn\^omes complexes. {P}artie {II}},
  volume 85-4 of {\em Publications Math\'ematiques d'Orsay [Mathematical
  Publications of Orsay]}.
\newblock Universit\'e{} de Paris-Sud, D\'epartement de Math\'ematiques, Orsay,
  1985.
\newblock With the collaboration of P. Lavaurs, Tan Lei and P. Sentenac.

\bibitem{Golomb}
Solomon~W. Golomb.
\newblock Irreducible polynomials, synchronization codes, primitive necklaces,
  and the cyclotomic algebra.
\newblock In {\em Combinatorial {M}athematics and its {A}pplications ({P}roc.
  {C}onf., {U}niv. {N}orth {C}arolina, {C}hapel {H}ill, {N}.{C}., 1967)},
  volume No. 4 of {\em University of North Carolina Monograph Series in
  Probability and Statistics}, pages 358--370. Univ. North Carolina Press,
  Chapel Hill, NC, 1969.

\bibitem{Hubbard-Schleicher}
John~H. Hubbard and Dierk Schleicher.
\newblock The spider algorithm.
\newblock In {\em Complex dynamical systems ({C}incinnati, {OH}, 1994)},
  volume~49 of {\em Proc. Sympos. Appl. Math.}, pages 155--180. Amer. Math.
  Soc., Providence, RI, 1994.

\bibitem{Miller}
Robert~L. Miller.
\newblock Necklaces, symmetries and self-reciprocal polynomials.
\newblock {\em Discrete Math.}, 22(1):25--33, 1978.

\bibitem{Milnor-Thurston}
John Milnor and William Thurston.
\newblock On iterated maps of the interval.
\newblock In {\em Dynamical systems ({C}ollege {P}ark, {MD}, 1986--87)}, volume
  1342 of {\em Lecture Notes in Math.}, pages 465--563. Springer, Berlin, 1988.

\bibitem{Reutenauer}
Christophe Reutenauer.
\newblock Mots circulaires et polyn\^omes irr\'eductibles.
\newblock {\em Ann. Sci. Math. Qu\'ebec}, 12(2):275--285, 1988.

\bibitem{Weiss-Rogers}
A.~Weiss and T.~D. Rogers.
\newblock The number of orientation reversing cycles in the quadratic map.
\newblock In {\em Oscillations, bifurcation and chaos ({T}oronto, {O}nt.,
  1986)}, volume~8 of {\em CMS Conf. Proc.}, pages 703--711. Amer. Math. Soc.,
  Providence, RI, 1987.

\end{thebibliography}

\end{document}